\documentclass[final,1p]{elsarticle}
\usepackage{amsmath,amssymb,amsthm}
\usepackage{txfonts}
\usepackage[pdftitle={S. Lord, D. Potapov, F. Sukochev -- Measures from Dixmier Traces and Zeta Functions},
            pdfauthor={S. Lord, D. S. Potapov, F. A. Sukochev},
            pdfsubject={2000MSC Primary: 46L51, 47B10, 58J42 Secondary: 46L87},
            pdfkeywords={Dixmier Trace, Zeta Functions, Noncommutative Integral ,
             						Noncommutative Geometry, Lebesgue Integral, Noncommutative Residue},
            pdfdisplaydoctitle=true,
            pdfpagelayout=SinglePage]{hyperref}

\numberwithin{equation}{section} 

\theoremstyle{plain}
\newtheorem{thm}{Theorem}[section]
\newtheorem*{thm*}{Theorem}
\newtheorem{prop}[thm]{Proposition}
\newtheorem*{prop*}{Proposition}
\newtheorem{lemma}[thm]{Lemma}
\newtheorem*{lemma*}{Lemma}
\newtheorem{cor}[thm]{Corollary}
\newtheorem*{cor*}{Corollary}
\theoremstyle{definition}
\newtheorem{dfn}[thm]{Definition}
\newtheorem*{dfn*}{Definition}
\newtheorem{rems}[thm]{Remark}
\newtheorem*{rems*}{Remark}
\newtheorem{ex}[thm]{Example}
\newtheorem*{ex*}{Example}
\newtheorem{cond}{Condition}
\newtheorem*{conj*}{Conjecture}
\newtheorem*{not*}{Notation}

\newcommand{\bl}{\ensuremath{\xi}}
\newcommand{\sqw}[2]{\ensuremath{\big< #1 \big>_{#2}}}
\newcommand{\abc}{\ensuremath{\! < \! \! < \!}}

\DeclareMathOperator{\Dom}{Dom}   
\DeclareMathOperator{\Tr}{Tr}     

\newcommand{\nm}[1]{\mbox{\ensuremath{\| #1 \|}}}
\newcommand{\fa}{\ensuremath{\forall}}

\newcommand{\inset}[2]{\ensuremath{\{ #1 \, | \, #2 \} }}
\newcommand{\inprod}[2]{\ensuremath{\langle #1 , #2 \rangle}}

\newcommand{\ceil}[1]{\ensuremath{\left\lceil #1 \right\rceil}}

\newcommand{\CC}{\ensuremath{\mathbb{C}}}
\newcommand{\RR}{\ensuremath{\mathbb{R}}}

\newcommand{\TT}{\ensuremath{\mathbb{T}}}

\newcommand{\ZZ}{\ensuremath{\mathbb{Z}}}

\newcommand{\NN}{\ensuremath{\mathbb{N}}}

\newlength{\indentation} 

\newcounter{propcount}

\newcounter{prop2count}
\newenvironment{prop2list}[3]{\begin{list}{\normalfont{(\roman{prop2count})}}
	{\usecounter{prop2count}
	\setlength{\topsep}{4pt} \setlength{\itemsep}{#3pt}
	\setlength{\parsep}{2pt} \setlength{\labelsep}{3mm}
	\setlength{\leftmargin}{#1mm}
	\setlength{\rightmargin}{#2mm}}}
       {\end{list}}

\newlength{\tablength} 

\begin{document}

\journal{arXiv.org}


\begin{frontmatter}

\title{Measures from Dixmier Traces and Zeta Functions}

\author[adl,syd]{Steven Lord\corref{cor1}\fnref{arc}}
\ead{steven.lord@adelaide.edu.au}
\author[syd]{Denis Potapov\fnref{arc}}
\ead{d.potapov@unsw.edu.au}
\author[syd]{Fedor Sukochev\fnref{arc}}
\ead{f.sukochev@unsw.edu.au}

\cortext[cor1]{Corresponding Author}
\fntext[arc]{Research supported by the Australian Research Council}

\address[adl]{School of Mathematical Sciences, University of Adelaide, Adelaide, 5005, Australia.}
\address[syd]{School of Mathematics and Statistics, University of New South Wales, Sydney, 2052, Australia.}

\begin{abstract} 

For $L^\infty$-functions on a (closed) compact Riemannian manifold, the noncommutative residue and the Dixmier trace formulation of the noncommutative integral are shown to equate to a multiple of the Lebesgue integral. 
The identifications are shown to continue to, and be sharp at, $L^2$-functions.  For functions strictly in $L^p$, $1 \leq p < 2$,
symmetrised noncommutative residue and Dixmier trace formulas must be introduced, for which the identification
is shown to continue for the noncommutative residue.
However, a failure is shown for the Dixmier trace formulation at $L^1$-functions.   It is shown the noncommutative residue remains finite and recovers the Lebesgue integral for any integrable function while the Dixmier trace expression can diverge.

The results show that a claim in the monograph J.~M.~Gracia-Bond{\'{i}}a, J.~C.~V{\'{a}}rilly and H.~Figueroa,
Elements of Noncommutative Geometry, Birkh\"{a}user, 2001, that 
the equality on $C^\infty$-functions between the Lebesgue integral
and an operator-theoretic expression involving a Dixmier trace (obtained from Connes' Trace Theorem) can be extended to any integrable function, is false.
The results of this paper include a general presentation
for finitely generated von Neumann algebras of commuting bounded operators, including a bounded Borel or $L^\infty$ functional calculus version of $C^\infty$ results in IV.2.$\delta$ A.~Connes, Noncommutative Geometry, Academic Press, New York, 1994.

\end{abstract}


\end{frontmatter}


\section{Introduction}

For a separable complex Hilbert space $H$, denote by $\mu_{n}(T)$, $n \in \NN$, the singular values of a compact operator $T$, (\cite{S}, \S 1).
Denote by $\mathcal{L}^1:= \mathcal{L}^1(H) = \inset{T}{\nm{T}_1:=\sum_{n=1}^\infty \mu_n(T) < \infty}$
the trace class operators.  It has long been known, see
(\cite{BR}, Thm 2.4.21 p.~76) (\cite{Ped}, Thm 3.6.4 p.~55), that a positive linear functional $\rho$ on a weakly closed
$^*$-algebra $\mathcal{N}$ of bounded operators on $H$ is normal (i.e.~$\rho$ belongs to the predual $\mathcal{N}_*$) if and only if
\begin{equation} \label{eq:tr}
\rho(A) = \Tr(AT) \ , \ A \in \mathcal{N}
\end{equation}
for a trace-class operator $0 < T \in \mathcal{L}^1$.  Denote by $\mathcal{L}^{1,\infty}:= \mathcal{L}^{1,\infty}(H) = \inset{T}{\nm{T}_{1,\infty}:=\sup_k \log(1+k)^{-1} \sum_{n=1}^k \mu_n(T) < \infty}$ the compact operators whose partial sums of singular values are logarithmically divergent.  
In \cite{Dix}, J.~Dixmier constructed a non-normal semifinite trace on the bounded linear operators of $H$ using the weight
$$
\Tr_\omega(T) := \omega \left( \left\{ \frac{1}{\log(1+k)} \sum_{n=1}^k \mu_n(T)\right\}_{k=1}^\infty  \right) \ , \ T>0
$$
associated to a translation and dilation invariant state $\omega$ on $\ell^{\infty}$.  As $\Tr_\omega$ vanishes on $\mathcal{L}^{1,\infty}_0 := \mathcal{L}^{1,\infty}_0(H) = \inset{T}{0 = \nm{T}_0 := \limsup_k \log(1+k)^{-1} \sum_{n=1}^k \mu_n(T)}$ and $\mathcal{L}^1 \subset \mathcal{L}^{1,\infty}_0$,
non-normality can be seen from $0 = \sup_{\alpha} \Tr_\omega(T_\alpha) \not= \Tr_\omega(1) =\infty$ for any strongly convergent sequence or net of finite rank operators $T_\alpha \nearrow 1$.
Fix $0 < T \in \mathcal{L}^{1,\infty}$ and let $B(H)$ denote the
bounded linear operators on $H$.  The weight
$$
\phi_{\omega}(A) := \Tr_\omega(AT) (= \Tr_\omega( \sqrt{T}A\sqrt{T} ) = \Tr_\omega( \sqrt{A} T \sqrt{A}) ) \ , \ 0 < A \in B(H)
$$
is finite and, by linear extension,
\begin{equation} \label{eq:Dtr}
\phi_{\omega}(A) = \Tr_\omega(AT) \ , \ A \in B(H) .
\end{equation}
From the properties of singular values, see (\cite{S}, Thm 1.6),
it follows $|\phi_\omega(A)| \leq \nm{A} \Tr_\omega(T)$, $A \in B(H)$.  Thus $\phi_\omega$ is a positive linear functional, i.e.~$\phi_\omega \in B(H)^*$.  While it is evident from preceding statements that $\phi_\omega \not\in B(H)_*$, it remains open on which proper weakly closed $^*$-subalgebras of $B(H)$ the functional $\phi_\omega$ is normal.
That there exist proper weakly closed $^*$-subalgebras $\mathcal{N} \subset B(H)$ with $\phi_\omega \in \mathcal{N}_*$ is part of the content of this paper.

Traditional noncommutative integration theory is based on normal linear functionals on von Neumann algebras, see \cite{Se2} and the monographs \cite{BR}, \cite{Ped}, \cite{T} (among many).  So it is somewhat surprising, and a disparity, that the formula (\ref{eq:Dtr}) with its obscured normality, and not (\ref{eq:tr}), appears as the analogue of integration in noncommutative geometry.  That it does is due to numerous results of
A.~Connes achieved with the Dixmier trace, see \cite{C3}, (\cite{CN}, \S IV), and \cite{CM} (as a sample).  In Connes' noncommutative geometry the formula (\ref{eq:Dtr}) has been termed the noncommutative integral, e.g.~(\cite{GBVF}, p.~297), (\cite{Hawk}, p.~478), due to the link to noncommutative residues in differential geometry described by the following theorem of Connes, see
(\cite{C3}, Thm 1), (\cite{GBVF}, Thm 7.18 p.~293).

\begin{thm}[Connes' Trace Theorem] \label{thm:2}
Let $M$ be a compact $n$-dimensional manifold, $\mathcal{E}$ a complex vector bundle
on $M$, and $P$ a pseudodifferential operator of order $-n$ acting on sections of $\mathcal{E}$. Then the corresponding operator $P$ in $H=L^2(M, \mathcal{E})$ belongs to $\mathcal{L}^{1,\infty}(H)$ and one has:
$$
\Tr_\omega(P) =  \frac{1}{n} \mathrm{Res}(P)
$$
for any $\omega$.
\end{thm}

Here $\mathrm{Res}$ is the restriction of the
Adler-Manin-Wodzicki residue to pseudodifferential operators of order $-n$, \cite{Wod}, \cite{C3}.  
Let $\mathcal{E}$ be the exterior bundle on a (closed) compact Riemannian manifold $M$, $|\mathrm{vol}|$ the 1-density of $M$ (\cite{GBVF}, p.~258),
$f \in C^\infty(M)$,
$M_f$ the operator given by $f$ acting by multiplication on smooth sections of $\mathcal{E}$,
$\Delta$ the Hodge Laplacian on smooth sections of $\mathcal{E}$, and $P = M_f (1+\Delta)^{-n/2}$, which is a pseudodifferential operator of order $-n$.
Using Theorem \ref{thm:2}, see (\cite{GBVF}, Cor 7.21), (\cite{FB}, \S 1.1), or (\cite{Landi}, p.~98),
\begin{equation} \label{eq:integral}
\phi_\omega(M_f) = \Tr_\omega(M_f T_\Delta) = \frac{1}{2^{(n-1)}\pi^{\frac{n}{2}}\Gamma(\frac{n}{2} + 1)} \int_{M} f(x) |\mathrm{vol}|(x) \ , \ f \in C^\infty(M)
\end{equation}
where we set $T_\Delta := (1+\Delta)^{-n/2} \in \mathcal{L}^{1,\infty}$.  This has become the standard way
to identify $\phi_\omega$ with the Lebesgue integral for $f \in C^\infty(M)$, see
\emph{op.~cit.}.  We note that in equation (\ref{eq:integral}), without loss, we can
assume the operators act on the Hilbert space $L^2(M)$ instead of
$L^2(M, \mathcal{E})$.  As mentioned above $\phi_\omega \in B(L^2(M))^*$.  The mapping
 $\phi: f \mapsto M_f$ is an isometric $^*$-isomorphism of $C(M)$, the continuous
 functions on $M$, into $B(L^2(M))$.  In this way $\phi_\omega \in C(M)^* \cong
 \phi(C(M))^*$ and, as the left hand side of (\ref{eq:integral}) is continuous in
 $\nm{\cdot}$ and the right hand side is continuous in $\nm{\cdot}_\infty$, the
 formula (\ref{eq:integral}) can be extended to $f \in C(M)$. 
 
\medskip The mapping $\phi: f \mapsto M_f$ is also an isometric $^*$-isomorphism of $L^\infty(M)$, the essentially bounded
functions on $M$, into $B(L^2(M))$.  In this way $\phi_\omega \in L^\infty(M)^* \cong \phi(L^\infty(M))^*$.  Extending the formula (\ref{eq:integral}) to $f \in L^\infty(M)$ has remained an elusive exercise however.  
Corollary 7.22 of (\cite{GBVF}, p.~297) made the claim that (\ref{eq:integral}) holds for any integrable function.  The short proof applied monotone convergence to both sides of (\ref{eq:integral}) to extend from $C^\infty$-functions to $L^\infty$-functions.
Monotone convergence can be applied to the right hand side, since the integral is a normal linear function on $L^\infty(M)$. 
To apply monotone convergence to the left hand side it must be known $\phi_\omega \in L^\infty(M)_*$.  The monograph \cite{GBVF} contained no proof that $\phi_\omega$ was normal.  Indeed, it is apparent from the next paragraph that the extension of (\ref{eq:integral}) to $f \in L^\infty(M)$ is equivalent to the statement $\phi_\omega \in L^\infty(M)_*$.

The task does not appear to be simplified by simplifying the manifold.  T.~Fack recently presented an argument that (\ref{eq:integral}) extends to $f \in L^\infty(\TT)$ for the 1-torus $\TT$, (\cite{Fack}, pp.~29-30).  The argument contains an oversight and provides the extension only for the first Baire class functions on the 1-torus\footnote{Private communication by P.~Dodds.}.
Fack's argument raises the point that $\phi \in L^\infty(\TT)^*$ is translation invariant (\cite{Fack}, p.~29), i.e. $\phi_\omega(M_{T_a(f)}) = \phi(M_{f})$ where $T_a(f)(x) = f(x+a)$, $x,a \in \TT$, is a translation operator.
Therefore $\phi_\omega$, when normalised, provides an invariant state on $L^\infty(\TT)$ that agrees (up to a constant) with the integral on $C(\TT)$.  Even this is not sufficient.  There are an infinitude of inequivalent invariant states on $L^\infty(\TT)$ which agree with the Lebesgue integral on $C(\TT)$ (\cite{Rudin}, Thm 3.4) (and first Baire class functions\footnote{We are indebted to B.~de Pagter for pointing this out and bringing Rudin's paper to our attention.  We also thank P.~Dodds for additional explanation.}).  The inequivalent states are non-normal
as the Lebesgue integral provides the only normal invariant state of $L^\infty(\TT)$
(uniqueness of Haar measure).

\medskip In this paper we show that $\phi_\omega(M_f)$, $f \in L^\infty(M)$, is identical to the Lebesgue integral up to a constant. For flat torii the method is elementary
and the Lebesgue integral can be recovered directly without recourse to Connes' Trace Theorem.  Primarily though, we investigate the claim of (\cite{GBVF}, Cor 7.22 p.~297) that the operator-theoretic formula $\phi_\omega(M_f)$ can be identified with the Lebesgue integral for any integrable function $f$ on a (closed) compact Riemannian manifold.  The claim is false.
We show the result is sharp at $L^2(M)$, indeed in Theorem \ref{thm:result2} (see also Examples~\ref{ex:sec4} and \ref{ex:compact_M})
we obtain $f \in L^2(M) \Leftrightarrow M_f(1+\Delta)^{-n/2} \in \mathcal{L}^{1,\infty}$, here $n$ is the dimension of the manifold. This type of sharp result at $L^2(M)$ for $M$ a compact manifold is
well-known, see for example Hausdorff-Young, Cwikel and Birman-Solomjak estimates in (\cite{S}, \S 4).

The sharp result leaves opens the question of whether some modified operator-theoretic formula can be identified with the Lebesgue integral for $f \in L^p(M)$, $1 \leq p < 2$.  
Calculating the Dixmier trace of $(1+\Delta)^{-n/2}$ using the residue of a zeta function originated in (\cite{CM}, p.~236).  Set $T_\Delta := (1+\Delta)^{-n/2}$.  We find in Theorem \ref{thm:1b2} that the residue at $s=1$ of the zeta function
$\Tr(T_\Delta^{s/2}M_{f}T_\Delta^{s/2})$, $s > 1$, equates to the Lebesgue integral of $f \in L^1(M)$ up to a constant.  Surprisingly, the Dixmier trace fails to equate to this residue.  We obtain the pointed result for flat torii that $\Tr_\omega( T_\Delta^{1/2} M_{f} T_\Delta^{1/2})$ equates to the Lebesgue integral of $f \in L^{1+\epsilon}(\TT^n)$, $\epsilon > 0$, yet there exists $f \in L^1(\TT)$ such that
$T_\Delta^{1/2} M_{f} T_\Delta^{1/2} \notin \mathcal{L}^{1,\infty}$, see Theorem \ref{TCLonePlusEpsilon} and Lemma \ref{lemma:6.3}.
In this sense,  not only is the claim of (\cite{GBVF}, Cor 7.22) false, its spirit has turned out to be false.  It is the noncommutative residue, not a Dixmier trace, which provides an algebraic formula completely identifying with the Lebesgue integral.

\medskip The structure of the paper is as follows. Preliminaries and the statement of the results mentioned above are given in Section \ref{sec:1}.  Section \ref{sec:1.1} introduces Dixmier traces.  Section \ref{sec:res.1} summarises known results on the calculation of a Dixmier trace using the zeta function of a compact operator.  Statements involving the Lebesgue integral on a (closed) compact Riemannian manifold appear in Section \ref{sec:2.5}.  

General statements
involving arbitrary finitely generated commutative von Neumann algebras and positive operators $D^2$, where $D=D^*$ has compact resolvent, appear in Theorem \ref{thm:1b} in Section \ref{sec:2.4}.  Conditions on the eigenfunctions of
$D^2$ and a set of selfadjoint commuting bounded operators $A_1,\ldots,A_n$ provide
\begin{equation} \label{eq:intronormal}
\phi_\omega(f(A_1,\ldots,A_n)) = \int_F f \circ e(x) v(x) d\mu(x) \ , \ \fa f \in L^\infty(E,\nu)
\end{equation}
for some $v \in L^1(F,\mu)$.  Here the von Neumann algebra generated by $A_1, \ldots, A_n$ is identified with a space of essentially bounded functions $L^\infty(E,\nu)$ on the joint spectrum $E$, $U: H \to L^2(F,\mu)$ is a spectral representation of $A_1, \ldots, A_n$,
$ \cdot \circ e$ is a normal embedding of $L^\infty(E,\nu)$ into $L^\infty(F,\mu)$, and $0 < T = G(D) \in \mathcal{L}^{1,\infty}$, $G$ a positive bounded Borel function, has Dixmier trace independent of $\omega$.  The characterisation (\ref{eq:intronormal})
implies $\phi_\omega$
is a \emph{unique} (independent of $\omega$) and \emph{normal} positive linear functional on the von Neumann algebra generated by $A_1,\ldots,A_n$.  
Section \ref{sec:ex} contains examples where $\phi_\omega$ can and cannot be characterised by (\ref{eq:intronormal}).

Section \ref{sec:tech} begins the technical results and contains the proof of Theorem \ref{thm:1b}.  Results of Section \ref{sec:tech} that may be of independent interest include: a generalised Cwikel or Birman-Solomjak type identity in Corollary \ref{cor:2.3}; a specialised extension of noncommutative residue formulations of the Dixmier trace in Theorem \ref{cor:2.3a}, and; normality
results in Section \ref{sec:3}.
Section \ref{sec:5} contains the proofs of the results in Section \ref{sec:2.5} and finishes the paper.

\medskip \noindent \textbf{Acknowledgements:} We thank Peter Dodds and Ben de Pagter for discussions concerning invariant means and Baire class functions.  The third named author thanks Thierry Fack and Bruno Iochum for useful discussions concerning Connes' Trace Theorem, and Airat Bikchentaev for direction to his papers.

\section{Statement of Main Results} \label{sec:1}

\subsection{Preliminaries on Dixmier Traces} \label{sec:1.1}

\medskip Let $\ceil{x}$, $x \geq 0$, denote the ceiling function.
Define the maps $\ell^\infty \to \ell^\infty$ for $j \in \NN$ by
\begin{eqnarray*}
T_j( \{a_k\}_{k=1}^\infty ) & = & \{ a_{k+j} \}_{k=1}^\infty \ , \
\{a_k\}_{k=1}^\infty \in \ell^\infty \ \\
D_j( \{a_k \}_{k=1}^\infty) & = & \{ a_{\ceil{j^{-1}k}} \}_{k=1}^\infty \ , \
\{a_k\}_{k=1}^\infty \in \ell^\infty .
\end{eqnarray*}
Set $BL := \inset{0 < \omega \in (\ell^{\infty})^*}{\omega(1) = 1, \omega \circ T_j = \omega \ \forall j \in \NN}$ (the set of Banach Limits)
and $DL := \inset{0 < \omega \in (\ell^{\infty})^*}{\omega(1) = 1, \omega \circ D_j = \omega \ \forall j \in \NN}$.
Both sets of states on $\ell^\infty$ satisfy
\begin{equation} \label{eq:genL}
\liminf_k a_k \leq \omega(\{ a_k \}_{k=1}^\infty) \leq \limsup_k a_k 
\end{equation}
for a positive sequence $a_k \geq 0$, $k \in \NN$.  Such states
are considered generalised limits, i.e.~extensions of $\lim$ on $c$ to $\ell^\infty$.
Let $0 < T \in \mathcal{L}^{1,\infty}$.  Set $\gamma(T) := \left\{ \log(1+k)^{-1} \sum_{n=1}^k \mu_n(T)\right\}_{k=1}^\infty \in \ell^\infty$
and define
\begin{multline*}
DL_2 := \{ 0 < \omega \in (\ell^{\infty})^* \, | \, \omega(1) = 1,
\omega \text{ satisfies }(\ref{eq:genL}), \\
\omega(D_2(\gamma(T))) = \omega(\gamma(T)) \ \forall 0 < T \in \mathcal{L}^{1,\infty} \} .
\end{multline*}
From (\cite{AF}, \S 5 Prop 5.2) or (\cite{CN}, pp.~303-308), for any $\omega \in DL_2$,
$$
\Tr_\omega(T) := \omega(\gamma(T)) \ , \ 0 < T \in \mathcal{L}^{1,\infty}
$$
defines a finite trace weight on $\mathcal{L}^{1,\infty}$ that vanishes on $\mathcal{L}^{1,\infty}_0$.
The linear extension, also denoted $\Tr_\omega$, is a finite trace on $\mathcal{L}^{1,\infty}$  that vanishes on $\mathcal{L}^{1,\infty}_0$.  Note
the condition that $\omega \in DL_2$ is weaker than the condition that
$\omega$ be dilation invariant, and weaker than Dixmier's original conditions, \cite{Dix}.

\subsection{Preliminaries on Residues of Zeta Functions} \label{sec:res.1}

\medskip A.~Connes introduced the association between a generalised zeta function,
$$
\zeta_T(s) := \Tr(T^s) = \sum_{n=1}^\infty \mu_{n}(T)^s \ , \ 0 < T \in \mathcal{L}^{1,\infty}
$$
and the calculation of a Dixmier trace with the result that
\begin{equation*} 
\lim_{s \to 1^+} (s-1) \zeta_T(s) = \lim_{N \to \infty} \frac{1}{\log(1+N)} \sum_{n=1}^{N} \mu_{n}(T)
\end{equation*}
if either limit exists, (\cite{CN}, p.~306).
Generalisations appeared in \cite{CRSS} and \cite{CPS}.  A short note, \cite{LS2}, authored by the first and third named authors, translated the results (\cite{CRSS}, Thm 4.11) and (\cite{CPS}, Thm 3.8) to $\ell^\infty$, see Theorem \ref{thm:resPA} and Corollary \ref{cor:resA} below.

We summarise the main result of \cite{LS2}, see \cite{CPS}, \cite{CRSS} and \cite{AF} for additional information. 
Define the averaging sequence $E : L^\infty([0,\infty)) \to \ell^\infty$ by
$$
E_k(f) :=  \int_{k-1}^k f(t)dt \ , \ f \in L^\infty([0,\infty)) .
$$
Define the map $L^{-1} : L^\infty([1,\infty)) \to L^\infty([0,\infty))$ by
$$
L^{-1}(g)(t) = g(e^t) \ , \ g \in L^\infty([1,\infty)).
$$
Define the piecewise mapping 
$
p : \ell^\infty \to L^\infty([1,\infty))
$
by
$$
p( \{a_k \}_{k=1}^\infty)(t) :=
\sum_{k=1}^\infty  a_k \chi_{[k,k+1)}(t) \ , \ \{ a_k \}_{k=1}^\infty \in \ell^\infty .
$$
Define, finally, the mapping $\mathcal{L} : (\ell^{\infty})^*
\to (\ell^{\infty})^*$ by
$$
\mathcal{L}(\omega) := \omega \circ E \circ L^{-1} \circ p \ , \
\omega \in (\ell^\infty)^*.
$$
We recall that $ T \in \mathcal{L}^{1,\infty}$ is called measurable (in the sense of Connes) if the value $\Tr_\omega(T)$ is independent of
$\omega \in DL_2$.  The equivalence between this definition of measurable and Connes' original (weaker) notion in (\cite{CN}, Def 7 p.~308) was shown in \cite{LSS}.

\begin{thm} \label{thm:resPA}
Let $P$ be a projection and $0 < T \in \mathcal{L}^{1,\infty}$.
Then, for any $\bl \in BL \cap DL$, $\mathcal{L}(\bl) \in DL_2$ and
$$
\Tr_{\mathcal{L}(\bl)}(PTP) = \bl \left( \frac 1k \Tr(PT^{1+\frac 1k}P) \right) .
$$
Moreover, $\lim_{k \to \infty} \frac 1k \Tr(PT^{1+\frac 1k}P)$ exists iff $PTP$ is measurable and in either case 
$$
\Tr_{\omega}(PTP) = \lim_{k \to \infty} \frac 1k \Tr(PT^{1+\frac 1k}P)
$$
for all $\omega \in DL_2$.
\end{thm}
\begin{proof}
See (\cite{LS2}, Thm 3.4).
\end{proof}

\begin{cor} \label{cor:resA}
Let $A \in B(H)$ and $0 < T \in \mathcal{L}^{1,\infty}$.
Then, for any $\bl \in BL \cap DL$,
$$
\Tr_{\mathcal{L}(\bl)}(AT) = \bl \left( \frac 1k \Tr(AT^{1+\frac 1k}) \right) .
$$
Moreover, $AT$ is measurable if $PTP$ is measurable for all  projections $P$ in the von Neumann algebra generated by $A$ and $A^*$. In this case,
$$
\Tr_{\omega}(AT) = \lim_{k \to \infty} \frac 1k \Tr(AT^{1+\frac 1k})
$$
for all $\omega \in DL_2$.
\end{cor}
\begin{proof}
See (\cite{LS2}, Cor 3.5).
\end{proof}

\subsection{Results for a Compact Riemannian Manifold} \label{sec:2.5}
 
\medskip Let $H$ be a separable complex Hilbert space and $D=D^*$ have compact resolvent.
Let $\{ h_m \}_{m=1}^\infty$ be a complete orthonormal system of eigenvectors of $D$
and $G(D)h_m = G(\lambda_m)h_m$ for any positive bounded Borel function $G$ where $\lambda_m$ are the eigenvalues of $D$.
Let $\bl \in BL \cap DL$ and $0 < G(D) \in \mathcal{L}^{1,\infty}$.  Then, from Corollary \ref{cor:resA},
$$
\Tr_{\mathcal{L}(\bl)}(AG(D)) = \bl \left( \frac 1k \sum_{m=1}^\infty G(\lambda_m)^{1+\frac 1k} \inprod{h_m}{Ah_m} \right) \ , \ A \in B(H) .
$$
As $\bl \in BL \cap DL$ vanishes on sequences converging to $0$, it follows that, for any $n \in \NN$,
$$
\Tr_{\mathcal{L}(\bl)}(AG(D)) = \bl \left( \frac 1k \sum_{m=n}^\infty G(\lambda_m)^{1+\frac 1k} \inprod{h_m}{Ah_m} \right) \ , \ A \in B(H) .
$$
Thus, for $A=A^*$ and $\bl \in BL \cap DL$,
$$
\inf_{m \geq n} \inprod{h_m}{Ah_m} \Tr_{\mathcal{L}(\bl)}(G(D)) \leq
\Tr_{\mathcal{L}(\bl)}(AG(D)) \leq \sup_{m \geq n} \inprod{h_m}{Ah_m} \Tr_{\mathcal{L}(\bl)}(G(D)) .
$$
Assuming $\Tr_{\mathcal{L}(\bl)}(G(D)) > 0$ and taking $n \to \infty$, we obtain the estimate
\begin{equation} \label{eq:estn}
\liminf_{m \to \infty} \inprod{h_m}{Ah_m} \leq
\frac{\Tr_{\mathcal{L}(\bl)}(AG(D))}{\Tr_{\mathcal{L}(\bl)}(G(D))}
 \leq \limsup_{m \to \infty} \inprod{h_m}{Ah_m}
 \ , \ A=A^* \in B(H)
\end{equation}
for any $\bl \in BL \cap DL$.

\begin{ex} \label{ex:ess}
Let $\TT^n$ be the flat $n$-torus, $\Delta$ be the Hodge Laplacian on $\TT^n$, and $0 < G(\Delta) \in \mathcal{L}^{1,\infty}$.
Then $h_{\mathbf{m}}(\mathbf{x}) = e^{i \mathbf{m} \cdot \mathbf{x}} \in L^2(\TT^n)$,
where $\mathbf{m} = (m_1,\ldots,m_n) \in \ZZ^n$ and $\mathbf{x} \in \TT^n$,
form a complete orthonormal system of eigenvectors of $\Delta$.
Let $M_f$ denote the operator of left multiplication of $f \in L^\infty(\TT^n)$ on $L^2(\TT^n)$, i.e.~$(M_fh)(\mathbf{x}) = f(\mathbf{x})h(\mathbf{x}) \ \fa h \in L^2(\TT^n)$.  Then
$$
\inprod{h_{\mathbf{m}}}{M_fh_{\mathbf{m}}} =
\int_{\TT^n} f(\mathbf{x}) d^n\mathbf{x} \ , \ f \in L^\infty(\TT^n)
$$
for all $\mathbf{m} \in \ZZ^n$.  Using the Cantor enumeration of
$\ZZ^n$, it follows from (\ref{eq:estn}) and for $\bl \in BL \cap DL$
that
\begin{equation} \label{eq:ess}
\Tr_{\mathcal{L}(\bl)}(M_fG(\Delta))
= \Tr_{\mathcal{L}(\bl)}(G(\Delta)) \int_{\TT^n} f(\mathbf{x}) d^n\mathbf{x} \ , \ f=\overline{f} \in L^\infty(\TT^n) .
\end{equation}
By linearity, (\ref{eq:ess}) holds for any $f \in L^\infty(\TT^n)$.
\end{ex}

The equality (\ref{eq:ess}) and the vanishing of $\Tr_{\mathcal{L}(\bl)}$ on $\mathcal{L}^1$ is, essentially, the proof of the following result for the flat torus $\TT^n$.

\begin{cor} \label{cor:result1}
Let $M$ be a $n$-dimensional (closed) compact Riemannian manifold with Hodge Laplacian $\Delta$. Set $T_\Delta := (1+\Delta)^{-n/2} \in \mathcal{L}^{1,\infty}(L^2(M))$.  Then 
$$
\phi_{\omega}(M_f) := \Tr_{\omega}(M_f T_\Delta) = c \int_{M} f(x) |\mathrm{vol}|(x) \ , \ \fa f \in L^\infty(M)
$$
where $c > 0$ is a constant independent of $\omega \in DL_2$.
\end{cor}

Complete details of the technicalities of the proof are contained in subsequent sections.  As mentioned, the Corollary is known
for $f \in C^\infty(M)$ from the application of Connes' Trace Theorem, see (\cite{FB}, p.~34). To our knowledge a proof for $f \in L^\infty(M)$ has not been given before. 
The main result is the extension to $L^2(M)$.

\begin{thm} \label{thm:result2}
Let $M$, $\Delta$, $T_\Delta$ be as in Corollary \ref{cor:result1}.  Then
$M_f T_\Delta \in \mathcal{L}^{1,\infty}(L^2(M))$ if and only if $f \in L^2(M)$ and 
$$
\phi_{\omega}(M_f) := \Tr_{\omega}(M_f T_\Delta) = c \int_{M} f(x) |\mathrm{vol}|(x) \ , \ \fa f \in L^2(M)
$$
where $c > 0$ is a constant independent of $\omega \in DL_2$.
\end{thm}

To our knowledge the if and only if statement in Theorem \ref{thm:result2} is new, although it is close
in spirit to the Hausdorff-Young, Cwikel and Birman-Solomjak estimates in (\cite{S}, \S 4).  As mentioned, the equalities were claimed as part of (\cite{GBVF}, Cor 7.22).
The proof of Theorem \ref{thm:result2} is in Section \ref{sec:5}.  It is more difficult to prove than Corollary \ref{cor:result1} as the condition $M_f T_\Delta \in \mathcal{L}^{1,\infty}$, for the \emph{unbounded} closable operator $M_f$, $f \in L^2(M)$, is non-trivial.  

With the if and only if statement, there exist
$f \in L^p(M)$, $1 \leq p < 2$, such that $M_fT_\Delta$
does not belong to the domain of any Dixmier trace.  To explore
any further identification between the Lebesgue integral
and an algebraic expression involving the Dixmier trace,
we considered the symmetrisation $T_\Delta^{1/2}M_fT_\Delta^{1/2}$ in the place of $M_fT_\Delta$.

For a compact linear operator $A > 0$, set
$\sqw{B}{A} := \sqrt{A}B\sqrt{A}$ for all linear operators $B$ such that $\sqw{B}{A}$ is densely defined and has bounded closure.  There are two situations when one uses the symmetrised expression $\sqrt{A}B\sqrt{A}$ instead of the product $AB$.  When $A \notin \mathcal{L}^{1,\infty}$
(as occurs in non-compact forms of noncommutative geometry), it is sometimes easier to obtain $\sqw{B}{A} \in \mathcal{L}^{1,\infty}$ than $BA \in \mathcal{L}^{1,\infty}$, see for example 
(\cite{GIV} \S 4.3).  A different use occurs when $B$ is unbounded, as
formulas such as $\Tr(\sqw{B}{A})$ may hold where $\Tr(AB)$ does not,
(\cite{Bik}, p.~163).  Our use is similar to the latter situation.

\begin{thm} \label{thm:1b2} Let $M$, $\Delta$, $T_\Delta$ be as in Corollary \ref{cor:result1}.  Then, $\sqw{M_f}{T_{\Delta}^s} = T_\Delta^{s/2} M_f T_\Delta^{s/2} \in \mathcal{L}^{1}(L^2(M))$ for all $s > 1$
if and only if $f \in L^1(M)$.  Moreover, setting
$$
\psi_\bl(M_f) := \bl \left( \frac{1}{k} \Tr(\sqw{M_f}{T_{\Delta}^{1+\frac{1}{k}}}) \right)
$$
for any $\bl \in BL$,
$$
\psi_\bl(M_f) := \lim_{k \to \infty} \frac{1}{k}
\Tr(\sqw{M_f}{T_{\Delta}^{1+\frac{1}{k}}}) = c \int_{M} f(x) |\mathrm{vol}|(x) \ , \ \fa f \in L^1(M)
$$
for a constant $c > 0$ independent of $\bl \in BL$.
\end{thm}

Thus $\psi_{\bl}$, as the residue of the zeta function
$\Tr(T_\Delta^{s/2} M_f T_\Delta^{s/2})$ at $s=1$, is the value of the Lebesgue integral of the integrable function $f$ on $M$.
This is the most general form of the identification between the
Lebesgue integral and an algebraic expression involving
$M_f$, the compact operator $(1+\Delta^2)^{-1}$ and a trace.

The claim of (\cite{GBVF}, Cor 7.22), which must use the symmetrised expression
for $f \in L^p(M)$, $1 \leq p < 2$, would be that $\Tr_\omega(\sqw{M_f}{T_\Delta}) = \Tr_\omega(T_\Delta^{1/2} M_f T_\Delta^{1/2})$
is also the Lebesgue integral of any integrable function.
Surprisingly, using an example on the flat torus,
we show this is false.

\begin{lemma} \label{lemma:counterex}
Let $\Delta$ be the Hodge Laplacian on the flat 1-torus $\TT$ and $T_\Delta = (1+\Delta)^{-1/2} \in \mathcal{L}^{1,\infty}(L^2(\TT))$.
There is a positive function $f \in L^1(\TT)$
such that the operator $T_\Delta^{1/2} M_f T_\Delta^{1/2}$ is not Hilbert-Schmidt.
\end{lemma}

This result is proven as Lemma~\ref{lemma:6.3} in Section \ref{sec:6}.  It says, in particular,
there exists $f \in L^1(\TT)$ such that $\overline{\phi}_{\omega}(M_f) = \infty \not= c \int_{\TT} f(x) dx$.
Our last result, proven as Theorem~\ref{TCLonePlusEpsilon},
shows this failure, at least for flat torii, is pointed
at $L^1$.

\begin{thm} \label{thm:1b3} Let $\Delta$ be the Hodge Laplacian on the flat n-torus $\TT^n$ and $T_\Delta = (1+\Delta)^{-n/2} \in \mathcal{L}^{1,\infty}(L^2(\TT^n))$.  If $f \in L^{1+\epsilon}(\TT^n)$,  $\epsilon > 0$, then $\sqw{M_f}{T_{\Delta}} = T_\Delta^{1/2} M_f
  T_\Delta^{1/2} \in \mathcal{L}^{1,\infty}(L^2(\TT^n))$.  Moreover,
$$
\overline{\phi}_{\omega}(M_f) := \Tr_{\omega}(\sqw{M_f}{T_{\Delta}}) = c \int_{\TT^n} f(\mathbf{x}) d^n \mathbf{x} \ , \ \fa f \in L^{1+\epsilon}(\TT^n)
$$
for a constant $c > 0$ independent of $\omega \in DL_2$.
\end{thm}



\subsection{Preliminaries on Joint Spectral Representations} \label{sec:prelim_sp}

\medskip Let $\mathcal{M}=\left< A_1,\ldots A_n \right>$ denote the von Neumann algebra generated by a finite set of selfadjoint commuting bounded operators $A_1,\ldots,A_n$ acting non-degenerately on $H$, i.e.~the weak closure of polynomials in $A_1, \ldots, A_n$.  Let $E$ denote the joint spectrum of $A_1,\ldots,A_n$.  Following (\cite{Ped} Thm 3.4.4),
let $\{ \eta_j \}_{j=1}^N$ be a maximal family of unit vectors in $H$ with $\overline{\mathcal{M}\eta_j} \cap \overline{\mathcal{M}\eta_k} = \{ 0 \}$, $j \not= k \in \{1,\ldots,N\}$, and $\oplus_{j=1}^N \overline{\mathcal{M} \eta_j} = H$.  Here $N$ may take the value $N = \infty$.  Define $\eta = \sum_{j=1}^N 2^{-j} \eta_j$ and $l_\eta(f) := \inprod{\eta}{f(A_1,\ldots,A_n)\eta}$ for all $f \in C(E)$.  From the Riesz-Markov Theorem (\cite{RS}, Thm IV.18 p.~111), $l_\eta$ is associated to a finite regular Borel measure $\mu_\eta$ and, as $\eta$ is cyclic for $\mathcal{M}$ on $\overline{\mathcal{M} \eta}$, $\mathcal{M} \cong L^\infty(E,\mu_\eta)$ (\cite{Ped}, Prop 3.4.3).
Without loss we may write $f(A_1,\ldots,A_n)$, $f \in L^\infty(E,\mu_\eta)$,
to denote an element of $\mathcal{M}$.  This description contains the continuous functional
calculus, $C(E) \subset L^\infty(E,\mu_\eta)$, and the bounded Borel functional
calculus $B(E) \subset L^\infty(E,\mu_\eta)$.

Now, let $U : H \to L^2(F,\mu)$ be a joint spectral representation
of $A_1,\ldots,A_n$ (\cite{RS}, p.~246) with $UA_iU^* = M_{e_i}$, $i =1,\ldots,n$,  for bounded
functions $e_i$ on $F$. Without loss, see (\cite{RS}, p.~227), we can take $F = \oplus_{j=1}^N \RR$ and
$$\mu(\oplus_{j=1}^N J_j) := \sum_{j=1}^N 2^{-j} \inprod{\eta_j}{\chi_{J_j}(A_1,\ldots,A_n) \eta_j},
$$
where $\chi_{J_j}$ is the characteristic function of $J_j \subset \RR$.
Define the mapping $e : F \to E$ by $x \mapsto (e_1(x),\ldots,e_n(x))$.  
It is immediate for $f \in B(E)$ that $Uf(A_1,\ldots,A_n)U^*  = M_{f \circ e}$ where
$f \circ e \in L^\infty(F,\mu)$.  It is not so immediate when $f \in L^\infty(E,\mu_\eta)$.  
We say $e$ is measure preserving if $\mu_\eta(e(J)) = 0 \Rightarrow \mu(J)=0$, $J$ a Borel subset of $F$.

\begin{prop}
Let $e$ be measure preserving.  Then $\cdot \circ e : L^\infty(E,\mu_\eta) \to L^\infty(F,\mu)$ is a normal $^*$-homomorphism. 
\end{prop}
\begin{proof}
Let $f \in [f]_{\mu_\eta}$ be a bounded function on $E$ representing
the equivalence class $[f]_{\mu_\eta} \in L^\infty(E,\mu_\eta)$.
Then $f \circ e(x)$ is a bounded function on $F$.
Take $g \in [f]_{\mu_\eta}$.
Now $(f - g)\circ e(J) \not= 0$ implies $\mu_\eta(e(J)) = 0$
which in turn implies $\mu(J) = 0$.
Hence $[f]_{\mu_\eta} \mapsto [f \circ e]_\mu$ is well defined.

Let $\pi_\eta^{-1}$ denote the $^*$-isomorphism
$L^\infty(E,\mu_\eta) \to \mathcal{M}$ and $M^{-1}$ denote the $^*$-isomorphism
$M_{[f]_\mu} \mapsto [f]_\mu$, $[f]_\mu \in L^\infty(F,\mu)$, see (\cite{Ped}, Prop 2.5.2).
As the map $U \cdot U^* : B(H) \to B(L^2(F,\mu))$ is strong-strong continuous,
$\cdot \circ e : [f]_{\mu_\eta} \mapsto M^{-1}(U\pi_\eta^{-1}([f]_{\mu_\eta})U^*)$
is a normal $^*$-homomorphism, (\cite{Ped}, \S 2.5.1).
\end{proof}

\begin{ex}
Suppose $\mathcal{M}$ has a cyclic vector $\eta \in H$.  Then $(E,\mu_\eta) \cong (F,\mu)$.
Recall that $\mathcal{M}$ has a cyclic vector for the separable Hilbert space $H$
if and only if $\mathcal{M}$ is maximally commutative (\cite{Ped}, Prop 2.8.3 p.~35).

As a particular example, take $A_i = M_{x_i}$ where $x_i$ are a finite number of co-ordinate functions for a compact Riemannian manifold $M$.  The function
$1 \in L^2(M)$ is a cyclic vector and $L^2(M)$ is a spectral representation
with $L^\infty(M) \cong \langle M_{x_i} \rangle$.  The function
$M \ni x \mapsto (x_1(x), \ldots, x_{np}(x)) \in \RR^{np}$ is measure preserving.  Here $n$ is the dimension of $M$ and $p$ the number of charts in a chosen atlas of $M$.
\end{ex}

\subsection{Dixmier Traces and Measures on the Joint Spectrum} \label{sec:2.4}

\medskip This section generalises the results for $L^\infty(M)$ and $\Delta$ to an arbitrary finitely generated commutative von Neumann algebra and positive operator $D^2$, where $D=D^*$ has compact resolvent, when certain conditions are met.  Besides providing succinct proofs for Section \ref{sec:2.5}, we feel the results of this section are of independent interest.

As in previous sections, let $H$ be a separable complex Hilbert space and $D=D^*$ have compact resolvent.  Let $\{ h_m \}_{m=1}^\infty \subset H$ be a complete orthonormal system of eigenvectors of $D$
and $G(D)h_m = G(\lambda_m)h_m$ for any positive bounded Borel function $G$ where $\lambda_m$ are the eigenvalues of $D$. Let $\mathcal{M}=\left< A_1,\ldots A_n \right>$ denote the von Neumann algebra generated by a finite set of selfadjoint commuting bounded operators $A_1,\ldots,A_n$ acting non-degenerately on $H$.
We assume -- see the preliminaries in Section \ref{sec:prelim_sp},

\begin{cond} \label{cond:1}
There is a normal $^*$-homomorphism $ \cdot \circ e : \mathcal{M} \cong L^\infty(E,\mu_\eta) \to L^\infty(F,\mu)$, where $E$ is the joint spectrum of $A_1,\ldots,A_n$ and $U : H \to L^2(F,\mu)$ is a
joint spectral representation.
\end{cond}

\begin{dfn} \label{dfn:spectral_meas}
Let $A_1, \ldots A_n$ be commuting bounded selfadjoint operators satisfying Condition \ref{cond:1}.   We say:
\begin{prop2list}{10}{0}{2}
\item $D$ is $(A_1,\ldots,A_n,U)$-\emph{dominated} if the modulus squared of the eigenfunctions
of $UDU^*$ are dominated by some $l \in L^1(F,\mu)$;
\item $G(D) \in \mathcal{L}^{1,\infty}$ is \emph{spectrally measurable} if, for all the projections $P \in U^*L^\infty(F,\mu)U$, $PG(D)P \in \mathcal{L}^{1,\infty}$
is measurable (in the sense of Connes).
\end{prop2list} 
\end{dfn}

Suppose $0 < G(D) \in \mathcal{L}^{1,\infty}$.  Then $0 < G(D)^{s} \in \mathcal{L}^1$, $\fa s > 1$, (\cite{CRSS} Thm 4.5(ii) p.~266).  By the formula (\ref{eq:tr})
\begin{equation} \label{eq:zeta_measures}
\zeta(A)(s) := \Tr(AG(D)^{s}) \ , \ A \in U^*L^\infty(F,\mu)U
\end{equation}
is a \emph{normal} positive linear functional on $U^*L^\infty(F,\mu)U \subset B(H)$ for any fixed $s > 1$.  Hence, for each $s > 1$,
there exists a Radon-Nikodym derivative
$v_s \in L^1(F,\mu)$ such that
$$
\zeta(f(A_1,\ldots,A_n))(s) = \int_F f \circ e(x) v_s(x) d\mu(x) \ , \ \fa f \in L^\infty(E,\mu_\eta) .
$$

\begin{thm} \label{thm:1b} Let $H$ be a separable Hilbert space
and $D=D^*$ have compact resolvent.  Let
$0 < G(D) \in \mathcal{L}^{1,\infty}$, $\omega \in DL_2$, and set
\begin{equation*}
\phi_{\omega}(\cdot) = \Tr_\omega( \cdot G(D)) .
\end{equation*}
Let $\{ A_1, \ldots, A_n \}$ be commuting
bounded selfadjoint operators acting non-degenerately on $H$ with joint spectral representation $U : H \to L^2(F,\mu)$ and joint spectrum $E$ such that $D$ is $(A_1,\ldots,A_n,U)$-dominated and
Condition \ref{cond:1} is satisfied.
Then
\begin{prop2list}{10}{0}{2}
\item $\phi_\omega \in \mathcal{M}_*$ and 
there exists $v_{G,\omega} \in L^1(F,\mu)$ such that 
$$
\phi_\omega(f(A_1,\ldots,A_n)) = \int_F f \circ e(x) v_{G,\omega}(x) d\mu(x) \ \ \fa f \in L^\infty(E,\mu_\eta),
$$
\item we have 
$$
\phi_{\omega}(f(A_1,\ldots,A_n)) = \int_F f \circ e(x) v(x) d\mu(x) \ \ \fa f \in L^\infty(E,\mu_\eta) ,
$$
where
$$
v = \lim_{k \to \infty} k^{-1} v_{1 + k^{-1}} \in L^1(F,\mu)
$$
if and only if $G(D)$ is spectrally measurable.  Here the limit is taken in the weak (Banach) topology $\sigma(L^1(F,\mu),L^\infty(F,\mu))$.
\end{prop2list}
\end{thm}

The proof of Theorem \ref{thm:1b} is in Section \ref{sec:4}.

\begin{rems}
Theorem \ref{thm:1b} has been presented in such a form as
to enable comparison with (\cite{CN}, \S IV Prop 15(b) p.~312).  In (\cite{CN}, \S IV Prop 15(b)) Connes associated the Dixmier
trace and the $C^\infty$-functional calculus of $A_1,\ldots,A_n$ to a measure on the joint spectrum.  Note
that the results of Theorem \ref{thm:1b} do not require Condition 1 if applied only to the bounded Borel functional calculus
of $A_1,\ldots,A_n$.  Condition 1 is required to identify $\mathcal{M}$ with a $L^\infty$-functional calculus.
\end{rems}

Theorem \ref{thm:1b} is, essentially, criteria for
$\phi_\omega \in \mathcal{M}_*$, i.e.~\emph{normality} of the functional $\phi_\omega$.  
Under these conditions the notion of noncommutative integral, Connes version, and notion of integral, Segal version, intersect.
It is therefore of interest to find examples where the criteria are satisfied, and $\phi_\omega$ is normal, and where the criteria fail and $\phi_\omega$ is not normal.

\section{Examples} \label{sec:ex}

\begin{ex} \label{ex:torus}
Let $\TT^n$ be the flat $n$-torus.  Let $U : L^2(\TT^n) \to L^2(\TT^n)$ be the trivial spectral representation of $L^\infty(\TT^n)$ (which is generated by the functions $e^{i\theta_j}$, $j=1,\ldots,n$).  Condition \ref{cond:1} is satisfied. Take the orthonormal basis
$h_{\mathbf{m}}(\mathbf{x}) = e^{i \mathbf{m} \cdot \mathbf{x}}$,
where $\mathbf{m} = (m_1,\ldots,m_n) \in \ZZ^n$ and $\mathbf{x} \in \TT^n$,
of eigenvectors of the Hodge Laplacian $\Delta$ on $\TT^n$.  Then
$|h_{\mathbf{m}}(\mathbf{x})|^2 = 1$ is dominated by $1 \in L^1(\TT^n)$.
The hypotheses of Theorem \ref{thm:1b} are satisfied.
\end{ex}

\begin{ex} \label{ex:counter1}
Take a selfadjoint operator $D$ on a separable Hilbert space $H$ with trivial kernel and compact resolvent such that $\nm{|D|^{-1}}_0 = \inf_{V \in \mathcal{L}^{1,\infty}_0} \nm{|D|^{-1} - V}_{1,\infty} = 1$.  For example $Dh_m = m h_m$ where $\{ h_m \}_{m=1}^\infty$ is an orthonormal basis of $H$.  Let $\mathcal{M}$ be the von Neumann algebra generated by $A_1 := |D|^{-1}$.  Clearly $[D,A_1] = 0$ and $\mathcal{M}$ contains the spectral projections of $D$.  Let $Q_{j}$ be the projection onto the $j^\mathrm{th}$-eigenvalue of $D$ and $Q'_N$ be the projection onto the first $N$ eigenvalues of $D$, where the eigenvalues are listed by increasing absolute value with repetition.  Then $P_N := \sum_{j=N}^\infty Q_j = 1 - Q'_N$, but $\liminf_N \nm{P_N |D|^{-1} P_N}_0 = \nm{|D|^{-1}}_0 > 0$.  By Proposition \ref{prop:suff}
below
$\phi_\omega( \cdot) := \Tr_\omega( \cdot |D|^{-1})$ is not normal for $\mathcal{M}$.
The hypotheses of Theorem \ref{thm:1b} cannot be fulfilled. 
Indeed, $U : H \to \ell^2$ given by $h_m \mapsto e_m := (\ldots, 0,1,0\ldots$), $1$ is in the $m^{\mathrm{th}}$-place, is the spectral representation of $A_1$ up to unitary equivalence.  Clearly the collection $\{ e_m \}$ cannot be dominated by any $l \in \ell^1$.
\end{ex}


\section{Technical Results} \label{sec:tech}

We establish notation that will remain in force for the rest of the document.  Thus,
$H$ denotes a separable complex Hilbert space and $D=D^*$ a selfadjoint operator with compact resolvent, $\{ h_m \}_{m=1}^\infty \subset H$ will denote an orthonormal basis
of eigenvectors of $D$ and $Dh_m = \lambda_mh_m$ the eigenvalues of $D$,
$G$ will denote a positive bounded Borel function such that $0 < G(D) \in \mathcal{L}^{1,\infty}$, $A_1,\ldots,A_n$ will denote a finite set of selfadjoint commuting bounded operators acting non-degenerately on $H$, and
$\mathcal{M}=\left< A_1,\ldots A_n \right>$ will denote the von Neumann algebra generated
by $A_1,\ldots,A_n$.

Condition \ref{cond:1} is assumed.
Without exception $U$ will denote the unitary $U : H \to L^2(F,\mu)$
such that $Uf(A_1,\ldots,A_n)U^* = M_{f \circ e}$ for all $f \in L^\infty(E,\mu_\eta)$, see Condition \ref{cond:1}.  Conversely, we
identify $T_f  := U^* M_f U \in B(H)$ for $f \in L^\infty(F,\mu)$.
Without exception, $(E,\mu_\eta)$ and $(F,\mu)$ will denote the respective measure spaces.

\subsection{Summability for Unbounded Functions} \label{sec:2.1}

\medskip Let $g : \RR \to \CC$ be a bounded Borel function. Set
\begin{equation} \label{eq:2.2}
\mathcal{F}_D(g)(x) := \sum_{m} g(\lambda_m) |(Uh_m)(x)|^2 .
\end{equation}
If $g(D) \in \mathcal{L}^1(H)$, the partial sums are Cauchy and convergence in the $L^1$-sense,
\begin{eqnarray*}
\int_{F} \left| \sum_{m=N}^M g(\lambda_m) |(Uh_m)(x)|^2 \right| d\mu(x)
& \leq & \sum_{m=N}^M |g(\lambda_m)| \int_F |(Uh_m)(x)|^2 d\mu(x) \\
& = & \sum_{m=N}^M |g(\lambda_m)| .
\end{eqnarray*}
Hence $\mathcal{F}_D(g) \in L^1(F,\mu)$ and $\nm{\mathcal{F}_D(|g|)}_1 = \nm{g(D)}_1$.
Let $\mu_g \abc \mu$ denote the (complex)
measure with Radon-Nikodym derivative $\mathcal{F}_D(g)$.
If $g(D) \in \mathcal{L}^s$ for $s \geq 1$, set $\mu_{s}$ to be the measure
with Radon-Nikodym derivative $\mathcal{F}_D(|g|^s)$.
If $g > 0$, $\mu_g \equiv \mu_1$.
In this section we relate summability of $T_fg(D)$ to the measures $\mu_g$ and $\mu_s$, $s \geq 1$.

\begin{lemma} \label{lemma:2.0}
Let $\{ f_n \}_{n=1}^\infty \subset L^\infty(F,\mu)$.
Suppose $f_n \to f$ pointwise $\mu$-a.e. such that $|f_n| \nearrow |f|$ and $\nm{f_n h}_2 \leq K$, $K > 0$, for $h \in L^2(F,\mu)$.
Then $\nm{f h}_2 \leq K$.
\end{lemma}
\begin{proof}
A simple application of Fatou's Lemma, since we obtain $\nm{f h}_2^2 = \int_F |f(x)|^2 |h(x)|^2 d\mu(x) \leq \sup_n \int_F |f_n(x)|^2|h(x)|^2 d\mu(x) \leq K^2$ from $|f_n|^2|h|^2 \nearrow |f|^2 |h|^2$ pointwise. 
\end{proof}

In the following Proposition and throughout the document,
the expression $T_fg(D)$ is bounded (or compact), where
$T_f$ is an unbounded closable
operator, refers
to the densely defined operator $T_fg(D)$ having bounded
(or compact) closure.

\begin{prop} \label{prop:2.1}
Let $g(D)$ be Hilbert-Schmidt.  Then $T_f g(D)$ is Hilbert-Schmidt if and only if $f \in L^2(F,\mu_2)$. 
\end{prop}
\begin{proof}
($\Leftarrow$) 
We first show $T_f g(D)$ is bounded.  Let $L^\infty(F,\mu) \ni f_n \to f$ pointwise with $|f_n| \nearrow |f|$.  Now
\begin{eqnarray*}
\nm{T_{f_n}g(D)h_m}^2
& = & |g(\lambda_m)|^2 \nm{T_{f_n}h_m}^2 \\
& = & |g(\lambda_m)|^2 \int_F |f_n(x)|^2|(Uh_m)(x)|^2 d\mu(x) \\
& = & \int_F |f_n(x)|^2 |g(\lambda_m)|^2 |(Uh_m)(x)|^2 d\mu(x) \\
& \leq & \nm{f_n}^2_{2,\mu_2} \leq \nm{f}^2_{2,\mu_2} .
\end{eqnarray*}
Applying the previous lemma, with $h := U(g(D)h_m)$ and $K := \nm{f}_{2,\mu_2}$, yields $\nm{T_fg(D)h_m} < \infty$.  Hence $h_m \in \text{Dom}(T_fg(D))$ for each $m$,
and $T_fg(D)$ is densely defined.

\medskip Now let $p_m$ be the one-dimensional projection onto $h_m$. 
Then $T_fg(D)p_m$ is one-dimensional.  Note that (*)
\begin{eqnarray*}
\nm{\sum_{m=1}^N T_fg(D)p_m}_2^2
& \stackrel{\text{(\cite{S}, Thm 1.18)}}{=} & \sum_k \nm{\sum_{m=1}^N T_fg(D)p_m h_k}^2 \\
& = & \sum_{m=1}^N \nm{g(\lambda_m) T_fh_m}^2 \\
& = & \sum_{m=1}^N |g(\lambda_m)|^2 \int_F |f(x)|^2 |(Uh_m)(x)|^2 d\mu(x) \\
& = & \int_F |f(x)|^2 \sum_{m=1}^N |g(\lambda_m)|^2 |(Uh_m)(x)|^2 d\mu(x) .
\end{eqnarray*}
This shows $\sum_{m=1}^N T_fg(D)p_m$ is a uniformly bounded sequence of bounded operators as
$$
\nm{\sum_{m=1}^N T_fg(D)p_m} \stackrel{\text{(\cite{S},  Thm  2.7(a))}}{\leq}  \nm{\sum_{m=1}^N T_fg(D)p_m}_2 \leq \nm{f}_{2,\mu_2} .
$$
The second inequality employed~(*). Let $h \in \text{Dom}(T_fg(D))$.  Then
\begin{eqnarray*}
\nm{T_fg(D)h} & = & \nm{\lim_{N \to \infty}
\sum_{m=1}^N T_fg(D)p_m h} \\
& \leq & \sup_N \nm{ \sum_{m=1}^N T_f g(D)p_m h} \leq \nm{f}_{2,\mu_2} \nm{h} .
\end{eqnarray*}
As $T_f g(D)$ is bounded on a dense domain,
$T_fg(D)$ has bounded closure.

\medskip Finally, now that it is established that (the closure) $T_fg(D)$ is bounded, by (*), the noncommutative Fatou Lemma and (\cite{S}, Thm 1.18), $T_fg(D) \in \mathcal{L}^2$ and $\nm{T_fg(D)}_2 = \nm{f}_{2,\mu_{2}}$.

\medskip ($\Rightarrow$) From (*), we can conclude $\int_F |f(x)|^2 \sum_{m=1}^N |g(\lambda_m)|^2 |(Uh_m)(x)|^2 d\mu(x)$ is a bounded increasing sequence.
Hence $\nm{f}_{2,\mu_2} < \infty$.

\end{proof}

\begin{cor} \label{cor:2.2}
Let $g(D) \in \mathcal{L}^1$.  Then:
\begin{prop2list}{12}{0}{2}
\item $T_fg(D) \in \mathcal{L}^1 \Rightarrow f \in L^2(F,\mu_{2})$;
\item $T_fg(D) \in \mathcal{L}^1 \Leftarrow  f \in L^2(F,\mu_{1})$.
\end{prop2list}
In both cases
$$
\Tr(T_fg(D)) = \int_F f(x) d\mu_{g}(x).
$$

\end{cor}
\begin{proof}
($\Rightarrow$) $g(D) \in \mathcal{L}^1$ implies $g(D) \in \mathcal{L}^2$
and $T_fg(D) \in \mathcal{L}^1$ implies $T_fg(D) \in \mathcal{L}^2$. 
Applying Proposition \ref{prop:2.1} shows $f \in L^2(F,\mu_2)$.

($\Leftarrow$) There exists $g_1$, $g_2$ such that $g_1g_2 = g$ and
$g_1(D)$ and $g_2(D)$ are Hilbert-Schmidt.  The function $\sqrt{|g|}$
can be chosen as $g_1$.  Then $T_fg_1(D)$ is Hilbert-Schmidt by
Proposition~\ref{prop:2.1} (note that measure~$\mu_2$ with respect
to~$g_1(D)$ coincide with measure~$\mu_1$ with respect to~$g(D)$).
Hence $T_fg_1(D)g_2(D) \in \mathcal{L}^1$.

The trace formula is evident from 
\begin{eqnarray*}
\Tr(T_fg(D)) & = & \sum_m \inprod{h_m}{T_fg(D)h_m} \\
& = & \sum_m g(\lambda_m) \int_F \overline{(Uh_m)(x)} f(x) (Uh_m)(x) d\mu(x) \\
& = & \int_F f(x) \sum_m g(\lambda_m) |(Uh_m)(x)|^2 d\mu(x) .
\end{eqnarray*}
\end{proof}

\begin{rems} \label{rem:RNder}
For $0 < G(D) \in \mathcal{L}^{1,\infty}$,
$\Tr(T_fG(D)^s) = \int_F f(x) d\mu_s$ by setting $g = G^s$, $s > 1$, in Corollary \ref{cor:2.2}.  From comparison with equation (\ref{eq:zeta_measures}) we have $v_s = \mathcal{F}_D(G^s) = d\mu_s/d\mu$, where $v_s$ are the Radon-Nikodym derivatives in
Theorem \ref{thm:1b} of Section \ref{sec:2.4}.  Notice immediately that
$\mu_s(F) = \Tr(G(D)^s)$, $s > 1$.
\end{rems}

\noindent We now fix $G$ such that $G(D) \in \mathcal{L}^{1,\infty}$
and, henceforth, $\mu_s \abc \mu$ is the measure with Radon-Nikodym
derivative $\mathcal{F}_D(|G|^s)$.  For $1 \leq p \leq \infty$, set
\begin{equation} \label{eq:def1weak}
L^p(F,\mu_{1,\infty}) := \inset{f}{f \in L^p(F,\mu_s), s > 1, \nm{f}_{1,\infty,p} < \infty}
\end{equation}
where
$$
\nm{f}_{1,\infty,p} := \sup_{1 < s \leq 2} (s-1)^{\frac{1}{p}} \nm{f}_{p,\mu_s}.
$$
Following (\cite{CRSS}, \S 4.2), for $T \in \mathcal{L}^{1,\infty}$ set
\begin{equation} \label{eq:Z_1}
\nm{T}_{Z_1} := \limsup_{s \to 1^+} (s-1)\Tr(|T|^s)^{\frac{1}{s}}.
\end{equation}
It was shown in (\cite{CRSS}, Thm 4.5) that $\nm{T}_0 \leq
e\nm{T}_{Z_1}$ and $\nm{T}_{Z_1} \leq \nm{T}_{1,\infty}$, where we
recall $\left\| T \right\|_0 = \inf_{V \in \mathcal{L}^{1,\infty}_0} \nm{T-V}_{1,\infty}$ is the Riesz seminorm on $\mathcal{L}^{1,\infty}$.

\begin{cor} \label{cor:2.3}
Let $G(D) \in \mathcal{L}^{1,\infty}$.  Then:
\begin{prop2list}{12}{0}{2}
\item $T_fG(D) \in \mathcal{L}^{1,\infty} \Rightarrow f \in L^2(F,\mu_{2})$;
\item $T_fG(D) \in \mathcal{L}^{1,\infty} \Leftarrow f \in L^2(F,\mu_{1,\infty})$.
\end{prop2list}
In case (ii), $\nm{T_fG(D)}_{Z_1} \leq \nm{f}_{1,\infty,2}\nm{G(D)}_{Z_1}^{1/2}$.
\end{cor}
\begin{proof}
($\Rightarrow$) $G(D) \in \mathcal{L}^{1,\infty}$ implies $G(D) \in \mathcal{L}^2$
and $T_f G(D) \in \mathcal{L}^{1,\infty}$ implies
$T_f G(D) \in \mathcal{L}^{2}$.  Apply Proposition \ref{prop:2.1}.

($\Leftarrow$) 
Without loss, assume $\nm{G(D)}=1$.  By (\cite{S}, p.~12), for $1 < s \leq 2$,
\begin{eqnarray*}
\nm{|T_fG(D)|^s}_{1}
& \leq & \sum_{m} \nm{T_fG(D)h_m}^{s} \\
& = & \sum_{m} \left( \int_F |f(x)|^2 |G(\lambda_m)|^2 |(Uh_m)(x)|^2 d\mu(x) \right)^{\frac s2} \\
& = &
\sum_{m} |G(\lambda_m)|^{\frac{(2-s)}{2}s} \left(\int_F |f(x)|^2 |G(\lambda_m)|^{s} |(Uh_m)(x)|^2 d\mu(x) \right)^{\frac{s}{2}} \\
& = & \sum_m A_m B_m
\end{eqnarray*}
where $A_m := |G(\lambda_m)|^{(2-s)s/2}$, 
$B_m := (\int_F |f(x)|^2 |G(\lambda_m)|^{s} |(Uh_m)(x)|^2 d\mu(x))^{s/2}$.
Set $\alpha := 2/(2-s)$ and $\beta := 2/s$.  It is clear
$\alpha^{-1} + \beta^{-1} = 1$.  Also note that
$\sum_m A_m^\alpha = \sum_m |G(\lambda_m)|^{s} < \infty$ for all $s > 1$.
Hence $\{ A_m \}_{m=1}^\infty \in \ell^\alpha$.  
For $B_m$,
$$
\sum_m B_m^{\beta} = \sum_{m} \int_F |f(x)|^2 |G(\lambda_m)|^{s} |(Uh_m)(x)|^2 d\mu(x) 
 =  \nm{f}_{2,\mu_{s}}^2 < \infty
$$
by (\ref{eq:def1weak}). Hence  $\{ B_m \}_{m=1}^\infty \in \ell^\beta$.
From the H\"{o}lder inequality
\begin{eqnarray*}
\nm{|T_fG(D)|^s}_{1} & \leq & \nm{\{A_m\}}_\alpha \nm{\{B_m \}}_\beta \\
& = & ( \Tr(|G(D)|^{s}) )^{\frac 1\alpha} ( \nm{f}_{2,\mu_{s}}^2 )^{\frac 1\beta} .
\end{eqnarray*}
Thus
\begin{equation} \label{eq:s-estimate}
\nm{T_fG(D)}_s  \leq \nm{G(D)}_s^{1-\frac{s}{2}} \nm{f}_{2,\mu_s} .
\end{equation}
Suppose $\nm{G(D)}_s \leq 1$, $s > 1$.  Then $\nm{G(D)}_{Z_1} = 0$
and, from (\ref{eq:s-estimate}),
$$
\nm{T_fG(D)}_{Z_1} = \limsup_{s \to 1^+} (s-1)\nm{T_fG(D)}_s  \leq \lim_{s\to 1^+} (s-1)^{\frac{1}{2}} \nm{f}_{1,\infty,2} = 0
$$
recalling $\nm{f}_{1,\infty,2} = \sup_{1 < s \leq 2} (s-1)^{1/2} \nm{f}_{2,\mu_s}$ from (\ref{eq:def1weak}).
By (\cite{CRSS}, Thm 4.5), $T_fG(D)$ belongs to $\mathcal{L}^{1,\infty}$.

Now, without loss, we can assume there is $s_0 > 1$ such that $\nm{G(D)}_{s_0} > 1$.  From $\nm{|G(D)|^s}_1 \geq \nm{|G(D)|^{s_0}}_1 > 1$ we have $\nm{G(D)}_s > 1$ for all $1 < s < s_0$. Under these assumptions $\nm{G(D)}_s^{1-s/2} \leq \nm{G(D)}_s^{1/2}$ for $1 < s < s_0$ and, from (\ref{eq:s-estimate}),
$$
(s-1)\nm{T_fG(D)}_{s} \leq \left((s-1)\nm{G(D)}_s \right)^{\frac{1}{2}} (s-1)^{\frac{1}{2}} \nm{f}_{2,\mu_{s}}
$$
for $1 < s < s_0$. This shows that
\begin{equation} \label{eq:1inf_est}
\nm{T_fG(D)}_{Z_1}
\leq \nm{f}_{1,\infty,2} \nm{G(D)}_{Z_1}^\frac{1}{2} < \infty .
\end{equation}
Again, by (\cite{CRSS}, Thm 4.5), $T_fG(D)$ belongs to $\mathcal{L}^{1,\infty}$.
\end{proof}

\begin{ex} \label{ex:sec4}
Let $\TT^n$ be the flat $n$-torus with $L^\infty(\TT^n)$, $L^2(\TT^n)$,
and $\Delta$, as in Example \ref{ex:torus}.  From the example, $\TT^n = E = F$, $\mu_\eta = \mu$ is Lebesgue measure and $M_f = T_f$.
Using the eigenfunctions of the Laplacian from Example \ref{ex:ess},
$\mathcal{F}_\Delta(|G|^s) = \Tr(|G(\Delta)|^s)$ (a constant).
Hence the measures $\mu_s$ associated to $\mathcal{F}_\Delta(|G|^s)$ are
multiples of Lebesgue measure.  In particular, for $T_\Delta = (1+\Delta)^{-n/2}$ we have, for any Borel set $J$,
$$
\mu_s(J) = \Tr(M_{\chi_J} T_\Delta^s ) = \Tr(T_\Delta^s) \mu(J).
$$
Here $\chi_J$ is the characteristic function of $J$.
Hence $\mu_s = \Tr(T_\Delta^s) \mu$, $s > 1$, which implies $\nm{\cdot}_{p,\mu_s} = \Tr(T_\Delta^s)^{1/p} \nm{\cdot}_p$
and $L^p(\TT^n,\mu_s) = L^p(\TT^n)$, $s > 1$.
Let $c := \sup_{1 < s \leq 2} (s-1) \Tr(T_\Delta^s)$, which is finite
as $T_\Delta \in \mathcal{L}^{1,\infty}$ (see, for example, Lemma \ref{lemma:cont_prelemma} below).
Then $\nm{\cdot}_{1,\infty,p} = c^{1/p} \nm{\cdot}_p$
and $L^p(\TT^n,\mu_{1,\infty}) = L^p(\TT^n)$.
We can conclude from Corollary \ref{cor:2.3} that $f \in L^2(\TT^n)$ if and only if
$M_f T_\Delta \in \mathcal{L}^{1,\infty}(L^2(\TT^n))$.
We also obtain, from the proof of Corollary \ref{cor:2.3}, that
$\nm{M_f T_\Delta}_{Z_1} \leq \nm{f}_2 \nm{T_\Delta}_{Z_1}$.
\end{ex}

\begin{ex} \label{ex:compact_M}
Let $M$ be a compact $n$-dimensional Riemannian manifold (without boundary) with Hodge Laplacian $\Delta$.
Let $\sigma_2(\Delta)$ denote the principal symbol
of the elliptic operator $\Delta$.  Locally
$\sigma_2(\Delta)(x_\alpha,\xi_\alpha) = g(x_\alpha)(\xi_\alpha,\xi_\alpha)$
where $x_\alpha = \phi_\alpha^{-1}(x)$ is a point in local co-ordinates
in a chart $(U_\alpha,\phi_\alpha)$ trivialising the tangent bundle,
$T_{x}(M) \cong \{x_\alpha \} \times \RR^n$, and $g(x_\alpha)$
is the matrix representation of the metric $g$
in the trivialisation.  Set
$\sqrt{|g|}(x_\alpha) = \sqrt{\det g(x_\alpha)}$.
The completely positive
pseudo-differential operator $T_\Delta := (1+\Delta^2)^{-n/2}$ is of order $-n$
and, from Connes' Trace Theorem (\cite{C3}), it belongs to $\mathcal{L}^{1,\infty}(L^2(M))$.

Let $h_m$ be an orthonormal basis of $L^2(M)$
and $f \in L^\infty(M)$.
Then, for $s > 1$,
\begin{eqnarray*}
\Tr(M_{f} T_\Delta^s)
& = & \sum_m \int_M f(x) \overline{h_m}(x)(T_\Delta^s h_m)(x) |\mathrm{vol}|(x) \\
& = & \int_M f(x) \left( \sum_m \overline{h_m}(x)(T_\Delta^s h_m)(x) \right) |\mathrm{vol}|(x) .
\end{eqnarray*}
We assume the volume 1-density is normalised.
For the flat torus the $L^1$-function
$$
k_s(x) = \sum_m \overline{h_m}(x)(T_\Delta^s h_m)(x)
$$
is a constant using the eigenvectors of the flat Laplacian.  
This will not be applicable in general.  In the general case
we require bounds on the function $k_s$.

Suppose $0 < c_s < k_s(x) < C_s \ \fa x \in M$.  Then we would have
$$
c_s^{\frac{1}{p}} \nm{f}_p \leq \nm{f}_{p,\mu_s}
:= \Tr(M_{|f|^p} T_\Delta^s)^{\frac{1}{p}}
\leq C_s^{\frac{1}{p}} \nm{f}_p
$$
for all $p \geq 1$.  So the $L^p$ norms
and the $\nm{\cdot}_{p,\mu_s}$ norms would be equivalent.

Let us examine the function $k_s$.  Let $P$ be a positive pseudo-differential operator of order $-ns$.
Let $y$ be a point in $U_\alpha$,
and $V_y \subset U_\alpha$ be a
rectangular neighbourhood of $y$.
For convenience we use $\phi_\alpha^{-1}(V_y) = \TT^n$,
the adjustment for the size of $V(y_\alpha)$
will not matter in the following argument
as $M$ is compact (the cover of $M$ by rectangular
neighbourhoods has a finite subcover).

Set $l_a$, $a \in \ZZ^n$, to be the function on
$M$ that is $e^{i x_\alpha \cdot a}$ in local co-ordinates on $\phi_\alpha^{-1}(V_y)$ and $0$ on $M \setminus V(y)$.
Note that
$$
\inprod{|g|^{-1/4}l_{a}}{|g|^{-1/4}l_{b}}
= \delta_{a,b} \chi_{V(y))} .
$$
On $V(y)$ we have the local Fourier decomposition
$$
h_m = \sum_{a} \inprod{|g|^{-1/4}l_a}{h_m}
|g|^{-1/4} l_a .
$$
Hence, for $x$ in the interior of $V(y)$,
$$
\overline{h_m}(x)(P h_m)(x)
= \sum_{a,b} \inprod{|g|^{-1/4}l_a}{h_m}
\inprod{h_m}{|g|^{-1/4}l_b}
|g|^{-1/4} \overline{l_b}(x) (P|g|^{-1/4}l_a)(x) .
$$
Then
\begin{eqnarray*}
\sum_m \overline{h_m}(x)(P h_m)(x)
& = & \sum_m \sum_{a,b} \inprod{|g|^{-1/4}l_a}{h_m}
\inprod{h_m}{|g|^{-1/4}l_b}
|g|^{-1/4} \overline{l_b}(x) (P|g|^{-1/4}l_a)(x) \\
& =  & \sum_{a,b} \langle |g|^{-1/4}l_a |
\left( \sum_m | h_m \rangle \langle h_m | \right)
| |g|^{-1/4}l_a \rangle
|g|^{-1/4} \overline{l_b}(x) (P|g|^{-1/4}l_a)(x_\alpha)) \\
& =  & \sum_{a} |g|^{-1/4} \overline{l_a}(x) (P|g|^{-1/4}l_a)(x) .
\end{eqnarray*}

Define the pseudo-differential
operator $P^{|g|}$ as
$|g|^{-1/4} P |g|^{-1/4}$.
Then, by definition of the symbol,
$$
|g|^{-1/4} \overline{l_a}(x) (P|g|^{-1/4}l_a)(x)
= \sigma(P^{|g|})(x_\alpha,a)
$$
up to some smooth term.  Hence, up to a smoothing term,
\begin{eqnarray}
\sum_m \overline{h_m}(x)(P h_m)(x)
& \approx &  \sum_{a} \sigma(P^{|g|})(x_\alpha,a) \label{eq:crux}.
\end{eqnarray}
The operator $P^{|g|}$ is of order
$-ns$, and, by the definition of a symbol of order $-ns$,
there is a constant $K_s$ (valid for all $x \in M$ as $M$ is compact)
such that $| \sigma(P^{|g|})(x_\alpha,a) | \leq K_s (1+\nm{a}^2)^{-ns/2}$.
This inequality holds with the addition of any smoothing term,
thus, from~\eqref{eq:crux},
\begin{eqnarray*}
\sum_m \overline{h_m}(x)(P h_m)(x)
& \leq &  K_s \sum_{a} (1+\nm{a}^2)^{-ns/2} =: C_s < \infty.
\end{eqnarray*}
Suppose $P = Q^s$, $s > 1$, where $Q$ is a positive pseudo-differential
operator of order $-n$.  That $Q$ is order $-n$
immediately implies there is a constant $K$, independent
of $s > 1$ such that $| \sigma((Q^s)^{|g|})(x_\alpha,a) | \leq K (1+\nm{a}^2)^{-ns/2}$.
Hence, for $1< s \leq 2$
\begin{equation} \label{eq:bound1}
\nm{f}_{p,\mu_s} \leq C^\frac{1}{p} \nm{f}_p
\end{equation}
for a constant $C : = K \sup_{1<s\leq2} (s-1)\sum_{a} (1+\nm{a}^2)^{-ns/2}$ independent of $s$, and 
\begin{equation} \label{eq:bound2}
\nm{f}_{1,\infty,p} \leq  C^\frac{1}{p} \nm{f}_p .
\end{equation}

Now suppose $Q$ is completely positive.  Then,
for any $0 < f \in L^\infty(M)$,
$0 < \Tr(M_fQ^s) = \int_M f(x) (\sum_m \overline{h_m}(x)(Q h_m)(x)) |\mathrm{vol}|(x)$.
Hence $k_s(x) = \sum_m \overline{h_m}(x)(Q h_m)(x) > 0$ almost everywhere.
However, from \eqref{eq:crux}, $k_s$ is identified with a smooth function in $x$.
Therefore, as $M$ is compact, $k_s$ attains some minimum value $c_s$.
Hence
\begin{equation} \label{eq:bound3}
c_s^{\frac{1}{p}} \nm{f}_p \leq \nm{f}_{p,\mu_s} .
\end{equation}

We can now apply the bounds~\eqref{eq:bound1}, \eqref{eq:bound2} and \eqref{eq:bound3}
to the completely positive pseudo-differential operator
$T_\Delta$ of order $-n$.
From Corollary \ref{cor:2.3} and \eqref{eq:bound2}, we have
\begin{equation} \label{eq:manifold_inequality_2}
\nm{M_f T_\Delta }_{Z_1} \leq C \nm{f}_2 \nm{T_\Delta}_{Z_1}^\frac{1}{2}.
\end{equation}
Using a sequence $L^\infty(M) \ni f_n \to f \in L^2(M)$
converging in the $L^2$-norm, we obtain
$M_f T_\Delta \in \mathcal{L}^{1,\infty}$ for all
$f \in L^2(M)$ and the inequality \eqref{eq:manifold_inequality_2} holds.
Moreover, if $M_f T_\Delta \in \mathcal{L}^{1,\infty}$
then $f \in L^2(M,\mu_2)$, which implies $f \in L^2(M)$ by \eqref{eq:bound3}.
Hence $M_f T_\Delta \in \mathcal{L}^{1,\infty}$
if and only if $f \in L^2(M)$.
\end{ex}


\subsection{Residues of Zeta Functions} \label{sec:extres}

\medskip In this section we extend the residue formulation of the noncommutative integral, see (\cite{CM}, App A), \cite{CPS}, \cite{CRSS}, to a specific class of unbounded functions.
As in (\ref{eq:def1weak}), for $1 \leq p \leq \infty$, set
\begin{equation*} 
L^p(F,\mu_{1,\infty}) := \inset{f}{f \in L^p(F,\mu_s), s > 1, \nm{f}_{1,\infty,p} < \infty}
\end{equation*}
where
$$
\nm{f}_{1,\infty,p} := \sup_{1 < s \leq 2} (s-1)^{\frac{1}{p}} \nm{f}_{p,\mu_s}.
$$
\begin{lemma} \label{lemma:cont_prelemma}
Let $G(D) \in \mathcal{L}^{1,\infty}$.  Then
$$
\sup_{1 < s \leq 2} (s-1)\mu_s(F) \leq \max \{ \nm{G(D)}_{1,\infty} , \nm{G(D)}_{1,\infty}^2 \}.
$$
\end{lemma}
\begin{proof}
From Remark \ref{rem:RNder}, $\mu_s(F) = \Tr(|G(D)|^s)$.
From the second last display of (\cite{CRSS}, p.~267),
$(s-1) \Tr(|G(D)|^s) \leq \nm{G(D)}_{1,\infty}^s$.
Then $\sup_{ 1< s \leq 2} \nm{G(D)}_{1,\infty}^s = \nm{G(D)}_{1,\infty}$ or $\nm{G(D)}_{1,\infty}^2$.
\end{proof}

For brevity, set $C := \max \{ \nm{G(D)}_{1,\infty} ,  \nm{G(D)}_{1,\infty}^2 \}$.

\begin{lemma} \label{lemma:cont_embed}
Let $q \geq p \geq 1$.  Then $L^q(F,\mu_{1,\infty})$ is continuously embedded in $L^p(F,\mu_{1,\infty})$.
In particular, $\nm{f}_{1,\infty,p} \leq C^{1/p - 1/q} \nm{f}_{1,\infty,q}$, $\fa f \in L^q(F,\mu_{1,\infty})$.
\end{lemma}
\begin{proof}
We recall, as $\mu_s$ is a finite measure on $F$, the standard embedding
$$
\nm{f}_{p,\mu_s} \leq \mu_s(F)^{\frac 1p - \frac 1q} \nm{f}_{q,\mu_s} .
$$
Hence
\begin{eqnarray*}
\nm{f}_{1,\infty,p} & = & \sup_{s > 1} (s-1)^{ \frac 1p } \nm{f}_{p,\mu_s} \\
& \leq & \sup_{s > 1} (s-1)^{\frac 1p - \frac 1q} \mu_s(F)^{\frac 1p - \frac 1q} (s-1)^{\frac 1q} \nm{f}_{q,\mu_s} \\
& \leq & C^{\frac 1p - \frac 1q} \nm{f}_{1,\infty,q} .
\end{eqnarray*}
\end{proof}

Denote by $L^p_0(F,\mu_{1,\infty}) \subset L^p(F,\mu_{1,\infty})$ the closure
of step functions on $F$ in the norm $\nm{\cdot}_{1,\infty,p}$.

\begin{lemma}
Let $1 \leq p \leq \infty$.  Then $L^\infty(F,\mu) \subset L_0^p(F,\mu_{1,\infty})$ and $\nm{f}_{1,\infty,p} \leq C^{1/p} \nm{f}_\infty$, $\fa f \in L^\infty(F,\mu)$.
\end{lemma}
\begin{proof}
If $f \in L^\infty(F,\mu)$, then $(s-1)^{1/p}\nm{f}_{p,\mu_s} \leq \nm{f}_\infty ((s-1)\mu_s(F))^{1/p} \leq \nm{f}_\infty C^{1/p}$.  Hence $L^\infty(F,\mu) \subset L^p(F,\mu_{1,\infty})$ for any $p$.  Let $f_n$ be step functions such that
$\nm{f - f_n}_\infty \to 0$ as $n \to \infty$.  Then 
$\nm{f-f_n}_{1,\infty,p} \leq \nm{f-f_n}_\infty C^{1/p}$. It follows $\nm{f-f_n}_{1,\infty,p} \to 0$ as $n \to \infty$.
\end{proof} 

From the lemmas we have the continuous embeddings,
$$
L^\infty(F,\mu) \subset L^q_0(F,\mu_{1,\infty}) \subset L^q(F,\mu_{1,\infty}) \subset L^p(F,\mu_{1,\infty}),
$$
for $q \geq p \geq 1$. 

\begin{thm} \label{cor:2.3a}
Let $0 < G(D) \in \mathcal{L}^{1,\infty}$ and $\bl \in BL \cap DL$.
Then 
$$
\phi_{\mathcal{L}(\bl)}(T_f) :=
\Tr_{\mathcal{L}(\bl)}(T_fG(D)) = 
\bl \left( \frac{1}{k} \int_F f(x) d\mu_{1+\frac{1}{k}}(x) \right) \ , \ \fa f \in L^2_0(F,\mu_{1,\infty}) .
$$
Moreover, if $\lim_{k \to \infty} k^{-1} \int_F h(x) d\mu_{1+k^{-1}}(x)$ exists for all $h \in L^\infty(F,\mu_{1,\infty})$, then
$$
\phi_{\omega}(T_f) :=
\Tr_{\omega}(T_fG(D)) = 
\lim_{k \to \infty} \frac{1}{k} \int_F f(x) d\mu_{1+\frac{1}{k}}(x) 
\ , \ \fa f \in L^2_0(F,\mu_{1,\infty})
$$
and all $\omega \in DL_2$.
\end{thm}
\begin{proof}
By hypothesis $f_n = \sum_j b_{n,j} \chi_{F_{n,j}} \to f$ where $F_{n,j} \subset F$ are Borel and disjoint, $\chi_{F_{n,j}}$ is the characteristic function of $F_{n,j}$, $b_{n,j} \in \CC$, the sum over $j$ is finite, and $\nm{f_n-f}_{1,\infty,2} \to 0$ as $n \to \infty$.  From Corollary \ref{cor:2.3} and (\cite{CRSS}, Thm 4.5),
$\nm{T_f G(D)}_0 \leq e \nm{f}_{1,\infty,2} \nm{G(D)}_{Z_1}^{1/2}$.  
Then, by construction,
\begin{equation} \label{prop:exteq1}
\left| \Tr_{\mathcal{L}(\bl)}((T_f-T_{f_n})G(D)) \right| \leq \nm{(T_f-T_{f_n})G(D)}_0
\stackrel{n}{\to} 0 . 
\end{equation}
By Corollary \ref{cor:2.2},
\begin{eqnarray*}
\bl \left( \left| \frac 1k \Tr((T_{f}-T_{f_n})G(D)^{1 +\frac 1k}) \right| \right)
& \leq & \bl \left( \frac 1k \int_F |(f-f_n)(x)| d\mu_{1 + \frac 1k}(x) \right) \\
& \leq & \sup_k  \frac 1k \nm{f-f_n}_{1,\mu_{1+k^{-1}}}  \\
& \leq & \nm{f-f_n}_{1,\infty,1}.
\end{eqnarray*}
From Lemma \ref{lemma:cont_embed}, $f_n$ converges to $f$ in $\nm{\cdot}_{1,\infty,1}$.  Hence
\begin{equation} \label{prop:exteq2}
\lim_{n \to \infty} \bl \left( \frac 1k \Tr((T_f-T_{f_n})G(D)^{1 +\frac 1k}) \right)
= 0 .
\end{equation}
Set the projection $P_{n,j} := T_{\chi_{F_{n,j}}}$.  Then
\begin{eqnarray}
\Tr_{\mathcal{L}(\bl)}(T_{f_n}G(D)) & = &
\Tr_{\mathcal{L}(\bl)}(\sum_j b_{n,j} P_{n,j}G(D)) \nonumber \\
& = & \sum_j b_{n,j} \Tr_{\mathcal{L}(\bl)}(P_{n,j}G(D)P_{n,j}) \nonumber \\
& \stackrel{\text{(Thm \ref{thm:resPA})}}{=} & \sum_j b_{n,j} \bl \left( \frac 1k \Tr(P_{n,j}G(D)^{1 +\frac 1k}P_{n,j}) \right) \nonumber \\
& = & \bl \left( \frac 1k \Tr(T_{f_n}G(D)^{1 +\frac 1k}) \right) .  \label{prop:exteq3}
\end{eqnarray}
If $\lim_{k \to \infty} k^{-1} \Tr(PG(D)^{1 + k^{-1}}P)$ exists
for all projections $P \in U^*L^\infty(F,\mu)U$,
then, by Theorem \ref{thm:resPA}, $\mathcal{L}(\bl)$ may be replaced in the preceding display by any $\omega \in DL_2$ and $\bl$ by $\lim$.
The results of the theorem follow from (\ref{prop:exteq1}), (\ref{prop:exteq2}) and (\ref{prop:exteq3}).
\end{proof}

\begin{ex} \label{ex:4.2}
Let $\TT^n$ be the flat $n$-torus with $L^\infty(\TT^n)$, $L^2(\TT^n)$,
and Hodge Laplacian $\Delta$, as in Examples \ref{ex:torus} and \ref{ex:sec4}.  Set $T_\Delta = (1+\Delta)^{-n/2}$.
From Example \ref{ex:sec4}, $M_f T_\Delta \in \mathcal{L}^{1,\infty}(L^2(\TT^n))$ iff $f \in L^2(\TT^n) (=L^2_0(\TT^n,\mu_{1,\infty})=L^2(\TT^n,\mu_{1,\infty}))$ and $\mu_s$
is a multiple of Lebesgue measure, $\mu_s = \Tr(T_\Delta^s)\mu$
for each $s > 1$.
From Theorem \ref{cor:2.3a}, for all $f \in L^2(\TT^n)$ and $\omega \in DL_2$,
\begin{eqnarray*}
\Tr_{\omega}(M_f T_\Delta) & = & \lim_{k \to \infty} \frac{1}{k} \int_{\TT^n} f(\mathbf{x}) \Tr(T_\Delta^{1+k^{-1}}) d^n\mathbf{x} \\
& = & \int_{\TT^n} f(\mathbf{x})d^n\mathbf{x} \lim_{k \to \infty} \frac{1}{k} \Tr(T_\Delta^{1+k^{-1}}) \\
& = & c \int_{\TT^n} f(\mathbf{x})d^n\mathbf{x}
\end{eqnarray*}
where $c = \lim_{k \to \infty} k^{-1} \Tr(T_\Delta^{1+k^{-1}}) = \lim_{s \to 1^+} (s-1) \Tr(T_\Delta^s) = \Tr_\omega(T_\Delta) < \infty$,
see (\cite{CM}, p.~236).
\end{ex}

\subsection{Sufficient Criteria for Normality} \label{sec:3}

\medskip Let $0 < G(D) \in \mathcal{L}^{1,\infty}$.  
Define $\nu_{G,\omega}: \text{Borel}(F) \to [0,\infty)$ for $\omega \in DL_2$
by
$$
\nu_{G,\omega}(J) := \Tr_\omega(T_{\chi_J}G(D)T_{\chi_J}) \ , \ \fa J \in \text{Borel}(F)
$$
where $\text{Borel}(F)$ denotes the Borel sets of $F$
and $\chi_J$ is the characteristic function of $J$.
We list sufficient criteria for $\nu_{G,\omega}$ to be a measure for all $\omega \in DL_2$.  

\begin{prop} \label{prop:suff0}
We have the following sequence of implications, \textrm{(i)} $\Rightarrow$ \textrm{(ii)}
$\Rightarrow$ \textrm{(iii)} $\Rightarrow$ \textrm{(iv)}:
\begin{prop2list}{12}{0}{2}
\item the sequence $\{ |Uh_m|^2 \}_{m=1}^\infty \subset L^1(F,\mu)$ is dominated by $l \in L^1(F,\mu)$;
\item for all collections of disjoint Borel sets $F_j \subset F$,
\begin{equation} \label{eq:meas}
\lim_{N \to \infty} \limsup_k \left( \frac 1k \sum_m G(\lambda_m)^{1+ \frac 1k} \int_{\cup_{j=N}^\infty F_j} |Uh_m(x)|^2 d\mu(x) \right)
= 0 ;
\end{equation}
\item for any sequence $Q_j$ of mutually orthogonal projections belonging to $U^*L^\infty(F,\mu)U$, $\nm{P_NG(D)P_N}_{0} \to 0$ as $N \to \infty$ where $P_N = \sum_{j=N}^\infty Q_j$;
\item $\nu_{G,\omega} \abc \mu$ is a finite Borel measure on $F$ for all $\omega \in DL_2$.
\end{prop2list}
\end{prop}
\begin{proof}
(i) $\Rightarrow$ (ii) By hypothesis $\int_{J} |(Uh_m)(x)|^2 d\mu(x)
\leq \int_{J} l(x) d\mu(x) =: \mu_l(J)$, where $\mu_l$ is the finite Borel measure on $F$ associated to $l$ and $J$ is a Borel set.  By countable additivity of $\mu_l$,
$\lim_{N \to \infty} \mu_l(\cup_{j=N}^\infty F_j)
= 0$.  Hence
\begin{eqnarray*}
\limsup_k k^{-1} \sum_{m} G(\lambda_m)^{1+ k^{-1}} \int_{ \cup_{j=N}^\infty F_j} |Uh_m(x)|^2 d\mu(x)
& \leq & \mu_l(\cup_{j=N}^\infty F_j) \limsup_k k^{-1} \sum_{m} G(\lambda_m)^{1+ k^{-1}} \\
& \leq & \mu_l(\cup_{j=N}^\infty F_j) \nm{G(D)}_{1,\infty} \to 0
\end{eqnarray*}
as $N \to \infty$.

(ii) $\Rightarrow$ (iii) From the first display in the proof of (\cite{CPS}, Prop 3.6 p.~88), it follows that $\limsup_k k^{-1}\Tr((PG(D)P)^{1+k^{-1}})
= \limsup_k k^{-1}\Tr(PG(D)^{1+k^{-1}}P)$ for all projections
$P \in B(H)$. By (\cite{CRSS}, Thm 4.5)
\begin{eqnarray*}
\nm{P_NG(D)P_N}_0 & \leq & e \limsup_k \frac 1k
\Tr((P_NG(D)P_N)^{1+ \frac 1k}) \\
& = & e \limsup_k \frac 1k
\Tr(P_NG(D)^{1+ \frac 1k}P_N) \\
& = & e \limsup_k \frac 1k \sum_{m} G(\lambda_m)^{1+ \frac 1k} \int_{ \cup_{j=N}^\infty F_j} |Uh_m(x)|^2 d\mu(x)
\end{eqnarray*}
where $Q_j = T_{\chi_{F_j}}$. (iii) now follows from (ii).

(iii) $\Rightarrow$ (iv) Set $P_N := \sum_{j=N}^\infty Q_j$ with $Q_j = T_{\chi_{F_j}}$.
Then $\Tr_{\omega}(P_NG(D)P_N) = \nu_{G,\omega}(\cup_{i=N}^\infty F_j)$.
Note that $\sup_{\omega \in DL_2} \Tr_{\omega}(P_NG(D)P_N) = \nm{P_NG(D)P_N}_0$ from (\cite{LSS}, Thm 6.4 p.~105), 
Hence, if $\nm{P_NG(D)P_N}_0 \to 0$ as $N \to \infty$, then $\nu_{G,\omega}(\cup_{i=N}^\infty F_j) \to 0$ as $N \to \infty$
for any $\omega \in DL_2$.  Thus $\nu_{G,\omega}$ is countably additive.
It is clear that, if $\mu(J) = 0$, $T_{\chi_{J}} = 0$ and hence 
$\nu_{G,\omega}(J)= \Tr_{\omega}(T_{\chi_{J}}G(D)T_{\chi_{J}}) = 0$.
This shows $\nu_{G,\omega} \abc \mu$.
\end{proof}

We recall again from (\cite{CN}, p.~308), \cite{LSS}, the notion of measurability.  We say
$0< G(D) \in \mathcal{L}^{1,\infty}$ is \emph{measurable} if $\Tr_\omega(G(D))$ is the same value for all $\omega \in DL_2$.  The first and third named authors with colleague
A.~Sedaev showed that measurability was equivalent to $\Tr_\omega(G(D)) = \lim_{N \to \infty} \log(1+N)^{-1} \sum_{n=1}^N \mu_n(G(D))$.
We say $G(D)$ is \emph{spectrally
measurable} (for the set $A_1,\ldots,A_n$ with joint spectral representation $U : H \to L^2(F,\mu)$) if $T_{\chi_J}G(D)T_{\chi_J}$ is measurable for all projections $\chi_J$ on $F$, see Definition \ref{dfn:spectral_meas}.  If $G(D)$ is spectrally measurable, $G(D)$ is measurable.  The converse is not true.

\begin{prop} \label{prop:suff1}
Let $G(D)$ be spectrally measurable with respect to the set $A_1,\ldots,A_n$ and the joint spectral representation $U : H \to L^2(F,\mu)$.
Then the statements (ii), (iii), (iv) in Proposition \ref{prop:suff0} are equivalent.
\end{prop}
\begin{proof}
We are required to show (iv) $\Rightarrow$ (ii).  By spectral
measurability there is a single measure,
\begin{eqnarray*}
\nu_{G,\omega}(J)
& = & \Tr_\omega(T_{\chi_J}G(D)T_{\chi_J}) \\
& \stackrel{\text{(Thm \ref{thm:resPA})}}{=} & \lim_{k \to \infty} k^{-1} \Tr(T_{\chi_J}G(D)^{1+k^{-1}}T_{\chi_J}) \\
& = & \limsup_k \left( \frac 1k \sum_m G(\lambda_m)^{1+ \frac 1k} \int_J |Uh_m(x)|^2 d\mu(x) \right)
\end{eqnarray*}
for a Borel set $J \subset F$.
The equation (\ref{eq:meas}) is obtained by setting $J = \cup^\infty_{j=N} F_j$ for disjoint Borel sets $F_j$ and taking $N \to \infty$.
\end{proof}

We now list some failure criteria
using the eigenvectors of $D$.  

\begin{prop} \label{prop:suff} \label{lemma:Ubnd}
Using the notation of Proposition \ref{prop:suff0},
if
\begin{equation*}
\liminf_{N \to \infty} \liminf_{m \to \infty} \inprod{h_m}{P_Nh_m} = \liminf_{N \to \infty} \liminf_m \int_{\cup_{j=N}^\infty F_j} |(Uh_m)(x)|^2 d\mu(x)
> 0
\end{equation*}
for some sequence of disjoint Borel sets $F_j$ (projections $P_N =
\sum_{j = N}^\infty T_{\chi_{F_j}}$), then $\nu_{G,\mathcal{L}(\bl)}$
is not a measure for any $\bl \in BL \cap DL$.
\end{prop}
\begin{proof}
From an identical argument for the estimate (\ref{eq:estn}), for any $\bl \in BL \cap DL$ and Borel set $J \subset F$,
\begin{multline*}
\liminf_m \int_J |Uh_m(x)|^2d\mu(x) \bl \left( \frac 1k \Tr(G(D)^{1 + \frac 1k}) \right) \\
\leq \ \   \nu_{G,\mathcal{L}(\bl)}(J)  \ \ \leq  \ \ \ \ \ \ \ \ \ \ \ \ \ \ \ \ \ \ \\
 \limsup_m \int_J |Uh_m(x)|^2d\mu(x)
\bl \left( \frac 1k \Tr(G(D)^{1 + \frac 1k}) \right) .
\end{multline*}
By this estimate and the hypothesis, $\nu_{G,\mathcal{L}(\bl)}$ is not countably additive.
\end{proof}

\subsection{Weak Convergence and Spectral Measurability} \label{sec:3.3} \label{sec:3.2}

\medskip We recall from, Remark \ref{rem:RNder}, the Radon-Nikodym derivatives
$v_s = \mathcal{F}_D(G^s) = d\mu_s / d\mu$, $s > 1$.

\begin{lemma} \label{thm:conv}
Let $0 < G(D) \in \mathcal{L}^{1,\infty}$.  If
$v := \lim_{k \to \infty} k^{-1} v_{1+k^{-1}}$
exists, where the limit is taken in the weak (Banach) topology
$\sigma(L^1(F,\mu),L^\infty(F,\mu))$, then $T_fG(D)$ is measurable and
$$
\Tr_\omega(T_fG(D)) = \int_F f(x) v(x)d\mu(x)
$$
for all $f \in L_0^2(F,\mu_{1,\infty})$ and $\omega \in DL_2$.
\end{lemma}
\begin{proof}
The assumption is $V_k := k^{-1} v_{1+k^{-1}}$ is a $\sigma(L^1(F,\mu),L^\infty(F,\mu))$-convergent sequence in $L^1(F,\mu)$
with limit $v$.  By the definition of weak convergence,
$$
\lim_{k \to \infty} \int_F f(x) V_k(x) d\mu(x) = \int_F f(x) v(x) d\mu(x)
$$
for all  $f \in L^\infty(F,\mu)$.  Then
$$
\lim_{k \to \infty} \left( \frac 1k \Tr(T_fG(D)^{1+\frac 1k}) \right)
 =  \lim_{k \to \infty} \int_F f(x) V_k(x) d\mu(x)
 =  \int_F f(x) v(x) d\mu(x)
$$
for all $f \in L^\infty(F,\mu)$.  It follows
$$
\Tr_{\omega}(T_{f}G(D)) = \lim_{k \to \infty} \int_F f(x) V_k(x) d\mu(x) =  \int_F f(x) v(x) d\mu(x)
$$
for all $f \in L_0^2(F,\mu_{1,\infty})$.  The first equality is from
the second part of Theorem \ref{cor:2.3a}.
\end{proof}

There is a partial converse.

\begin{lemma} \label{cor:partconv}
Suppose $D$ is $(A_1,\ldots,A_n,U)$-dominated and $0 < G(D) \in \mathcal{L}^{1,\infty}$ is spectrally measurable (see Definition \ref{dfn:spectral_meas}).
Then $v := \lim_{k \to \infty} k^{-1} v_{1+k^{-1}}$ exists, where the limit is taken in the weak (Banach) topology
$\sigma(L^1(F,\mu),L^\infty(F,\mu))$.
\end{lemma}
\begin{proof}
  Set $V_k := k^{-1} v_{1+k^{-1}}$.  By the proof of
  Proposition~\ref{prop:suff1} there exists a unique measure
  (independent of $\omega \in DL_2$)
\begin{eqnarray*}
\nu_{G,\omega}(J)
& = & \Tr_\omega(T_{\chi_J}G(D)T_{\chi_J}) \\
& = & \lim_{k \to \infty} k^{-1} \Tr(T_{\chi_J}G(D)^{1+k^{-1}}T_{\chi_J}) \\
& = & \lim_{k \to \infty} \int_J V_k(x) d\mu(x),
\end{eqnarray*}
for a Borel set $J$ of $F$.  Let $v$ be the Radon-Nikodym derivative of
$\nu_{G,\omega}$. 
Then,
\begin{equation} \label{eq:15weak}
\lim_{k \to \infty} \int_J (v(x) - V_k(x)) d\mu(x) = 0 .
\end{equation}
Equation (\ref{eq:15weak}) implies $\sigma(L^1(F,\mu),L^\infty(F,\mu))$-convergence.
\end{proof}

\subsection{Proof of Theorem \ref{thm:1b}} \label{sec:4}

With the technical results of the previous sections, we are in a position to prove Theorem \ref{thm:1b} (and Theorem \ref{thm:result2} in the next section).

\medskip (i) \ By the hypothesis that $D$ is $(A_1,\ldots,A_n,U)$-dominated, it follows 
from Proposition \ref{prop:suff0} that
$\nu_{G,\omega} \abc \mu$ is a finite Borel measure.  Let
$v_{G,\omega}$ be the Radon-Nikodym derivative of $\nu_{G,\omega}$.
Let $f \in L^\infty(F,\mu)$.  Take a sequence of step functions
$f_{n} := \sum_{i=1}^{N_n} a_{n,i} \chi_{F_{n,i}} \to f$ in norm.  Then $T_{f_n} \to T_f$ in the uniform norm and
\begin{eqnarray*}
\int_F f(x) v_{G,\omega} d\mu(x) & = & \lim_{n \to \infty} \int_F f_n(x) v_{G,\omega} d\mu(x) \\
& = & \lim_{n \to \infty} \sum_{i=1}^{N_n} a_{n,i} \nu_{G,\omega}(\chi_{F_{n,i}}) \\
& = & \lim_{n \to \infty} \sum_{i=1}^{N_n} a_{n,i} \Tr_{\omega}(T_{\chi_{F_{n,i}}}G(D)) \\
& = & \lim_{n \to \infty}  \Tr_{\omega}(\sum_{i=1}^{N_n} a_{n,i} T_{\chi_{F_{n,i}}}G(D)) \\
& = & \lim_{n \to \infty}  \phi_{\omega}(T_{f_n}) \\
& = & \phi_{\omega}(T_{f}) 
\end{eqnarray*}
by $\phi_\omega \in B(H)^*$.  Finally, if $f \in L^\infty(E,\mu_\eta)$, by Condition \ref{cond:1}, $f \circ e \in L^\infty(F,\mu)$.  It follows from the identification
of $\phi_\omega$ with the measure $\nu_{G,\omega} \abc \mu$ that $\phi_\omega \in \mathcal{M}_*$.

\medskip \noindent (ii) \ The if and only if statement is contained in
Lemma \ref{thm:conv} and Lemma \ref{cor:partconv}. The equality in
Lemma \ref{thm:conv} holds
for any $f \in L^\infty(F,\mu)$.
Finally, if $f \in L^\infty(E,\mu_\eta)$, by Condition \ref{cond:1},
$f \circ e \in L^\infty(F,\mu)$.
\qed

\section{Proofs for Compact Riemannian Manifolds} \label{sec:5}

Let $\TT^n$ be the flat $n$-torus and $\Delta$ be the Hodge Laplacian on $\TT^n$. In this situation
$h_{\mathbf{m}}(\mathbf{x}) = e^{i \mathbf{m} \cdot \mathbf{x}} \in L^2(\TT^n)$,
where $\mathbf{m} = (m_1,\ldots,m_n) \in \ZZ^n$ and $\mathbf{x} \in \TT^n$,
form a complete orthonormal system of eigenvectors of $\Delta$.
Let $M_f$ denote the operator of left multiplication of $f \in L^p(\TT^n)$ on $L^2(\TT^n)$, $1 \leq p \leq \infty$, i.e.~$(M_fh)(\mathbf{x}) = f(\mathbf{x})h(\mathbf{x})$
for all $h \in \Dom(M_f)$ (dense in $L^2(\TT^n)$).  Stronger results
than Theorem~\ref{thm:result2} are possible for the torus.

\begin{cor} \label{cor:3.2}
Let $g(\Delta) \in \mathcal{L}^1(L^2(\TT^n))$.  Then $M_f g(\Delta) \in \mathcal{L}^1(L^2(\TT^n))$ if and only if
$f \in L^2(\TT^n)$ and
$$
\Tr(M_fg(\Delta)) = \Tr(g(\Delta)) \int_{\TT^n} f(\mathbf{x}) d^n \mathbf{x} \ , \ \fa f \in L^2(\TT^n) .
$$
\end{cor}
\begin{proof}
  The corollary follows if Corollary~\ref{cor:2.2} is applied to
  Example~\ref{ex:sec4}.
\end{proof}

\begin{cor} \label{cor:3.3}
Let $0 < G(\Delta) \in \mathcal{L}^{1,\infty}(L^2(\TT^n))$ be measurable.  Then $M_f G(\Delta) \in \mathcal{L}^{1,\infty}(L^2(\TT^n))$ if and only if $f \in L^2(\TT^n)$ and
$$
\phi_\omega(M_f) := \Tr_{\omega}(M_fG(\Delta)) = c \int_{\TT^n} f(\mathbf{x}) d^n \mathbf{x} \ , \ \fa f \in L^2(\TT^n)
$$
where $0 \leq c = \Tr_{\omega}(G(\Delta))$ is a constant for all $\omega \in DL_2$.
\end{cor}
\begin{proof}
  The if and only if result is immediate from Example~\ref{ex:sec4}
  and Corollary~\ref{cor:2.3}.  The equality was shown in Example~\ref{ex:4.2} where $T_\Delta$ is replaced, without loss, by $G(\Delta)$.
\end{proof}

\subsubsection*{Proof of Theorem \ref{thm:result2}}

\medskip The statement $M_f T_\Delta \in \mathcal{L}^{1,\infty}$ if and only
if $f \in L^2(M)$ is contained in
Example~\ref{ex:compact_M}.

Also from Example~\ref{ex:compact_M} is the inequality \eqref{eq:manifold_inequality_2},
$$
|\Tr_\omega(M_f T_\Delta)|
\leq \nm{M_f T_\Delta}_{Z_1}
\leq C^{1/2} \nm{f}_2 \nm{M_f T_\Delta}_{Z_1}^{1/2} .
$$
If $C^\infty(M) \ni f_n \to f \in L^2(M)$
in the $L^2$-norm (also in the $L^1$-norm as $M$ is compact),
then, using the above inequality and Connes'
Trace Theorem,
\begin{eqnarray*}
\Tr_\omega(M_{f} T_\Delta)
& = & \lim_{n \to \infty} \Tr_\omega(M_{f_n} T_\Delta) \\
& = & \lim_{n \to \infty} c \int_M f_n(x) |\mathrm{vol}(x)| \\
& = & c \int_M f(x) |\mathrm{vol}(x)| .
\end{eqnarray*}
\qed

\medskip Corollary \ref{cor:result1} is an immediate corollary of Theorem \ref{thm:result2}.

\subsection{Extending to $L^1$} \label{sec:6}

\medskip The sharp result $M_f G(\Delta) \in \mathcal{L}^{1,\infty}(L^2(\TT^n)) \Leftrightarrow f \in L^2(\TT^n)$ in Corollary \ref{cor:3.3} is the extent of the identification between
$\phi_\omega(M_f)$ and the Lebesgue integral of $f$.  We investigate
extensions of the formula $\phi_\omega$ using the symmetrised expression
$G(\Delta)^{1/2} M_{f} G(\Delta)^{1/2}$ in place of $M_f G(\Delta)$.

Let us first demonstrate some properties of the symmetrised expression. For a compact linear operator $A > 0$, set
$\sqw{B}{A} := \sqrt{A}B\sqrt{A}$ for all linear operators $B$ such that $\sqw{B}{A}$ is densely defined on $H$ and has bounded closure.  

\begin{lemma} \label{lemma:sym}
Suppose $B > 0$ and $p \geq 1$.  Then $\sqrt{A}B\sqrt{A} \in \mathcal{L}^p$ (resp. $\mathcal{L}^{1,\infty}$) if and only if $\sqrt{B}A\sqrt{B} \in \mathcal{L}^p$ (resp. $\mathcal{L}^{1,\infty}$).  Moreover, if either condition holds, $\Tr((\sqrt{A}B\sqrt{A})^p) = \Tr((\sqrt{B}A\sqrt{B})^p)$
(resp. $\Tr_\omega(\sqrt{A}B\sqrt{A}) = \Tr_\omega(\sqrt{B}A\sqrt{B})$ for $\omega \in DL_2$).
\end{lemma}
\begin{proof}
Note $\sqrt{B}A\sqrt{B} = |\sqrt{A}\sqrt{B}|^2$
and $\sqrt{A}B\sqrt{A} = |\sqrt{B} \sqrt{A}|^2$. Now
$|\sqrt{A} \sqrt{B}|^2 \text{ compact } \Leftrightarrow  \sqrt{A}\sqrt{B} \text{ compact } \Leftrightarrow  \sqrt{B}\sqrt{A} = (\sqrt{A}\sqrt{B})^* \text{ compact } \Leftrightarrow  |\sqrt{B}\sqrt{A}|^2 \text{ compact}$.
All results follow since $\sqrt{A}\sqrt{B}$ and $\sqrt{B}\sqrt{A} = (\sqrt{A}\sqrt{B})^*$ have the same singular values (\cite{S}, p.~3).
See also \cite{Bik98} and references therein.
\end{proof}


\begin{prop} \label{prop:4.1}
Let $0 < g(D) \in \mathcal{L}^1$ and use the notation of Section \ref{sec:tech}.  Then $\sqw{T_{|f|}}{g(D)} \in \mathcal{L}^1$ if and only if $f \in L^1(F,\mu_g)$.  In both cases 
$$
\Tr(\sqw{T_f}{g(D)}) := \Tr(g(D)^{1/2} T_f g(D)^{1/2}) = \int_F f(x) d\mu_{g}(x)
$$
and $\nm{f}_{1,\mu_g} = \nm{\sqw{T_{|f|}}{g(D)}}_1$.
\end{prop}
\begin{proof}
Note that $\sqrt{g}(D) \in \mathcal{L}^2$ since $g(D) \in \mathcal{L}^1$.  Let $f > 0$.
Then $\sqrt{g}(D) T_{f} \sqrt{g}(D) \in \mathcal{L}^1 \Leftrightarrow 
T_{\sqrt{f}} \sqrt{g}(D) \in \mathcal{L}^2 \Leftrightarrow \sqrt{f} \in L^2(F,\mu_1)$.  The first equivalence is from the workings of the last lemma.  The second equivalence follows from Proposition
\ref{prop:2.1}.  Note, when applying the Proposition, that $\mu_{2}$ associated to $\sqrt{g}$ is equivalent to $\mu_1 = \mu_g$ associated to $g$.
If $f \in L^1(F,\mu_g)$ is not positive, $|f| \in L^1(F,\mu_g)$, hence $\sqw{T_{|f|}}{g(D)} \in \mathcal{L}^1$.  If $f$ is not
positive but $\sqw{T_{|f|}}{g(D)} \in \mathcal{L}^1$, then $|f| \in L^1(F,\mu_g)$.  Hence $f \in L^1(F,\mu_g)$.  Note, if $f \in L^1(F,\mu_g)$, then $f$ is a linear combination of
four positive integrable functions.  By linearity $\sqw{T_{f}}{g(D)} \in \mathcal{L}^1$.
The trace formula is evident from 
\begin{eqnarray*}
\Tr(\sqw{T_f}{g(D)}) & = & \sum_m \inprod{\sqrt{g}(D) h_m}{T_f \sqrt{g}(D)h_m} \\
& = & \sum_m g(\lambda_m) \int_F \overline{(Uh_m)(x)} f(x) (Uh_m)(x) d\mu(x) \\
& = & \int_F f(x) \sum_m g(\lambda_m) |(Uh_m)(x)|^2 d\mu(x) .
\end{eqnarray*}
\end{proof}

It is now easy to extend Corollary \ref{cor:3.2} and Corollary \ref{cor:3.3}
in the case of the flat $n$-torus $\TT^n$ and Hodge Laplacian $\Delta$.

\begin{cor} \label{cor:4.2}
Let $0 < g(\Delta) \in \mathcal{L}^1(L^2(\TT^n))$.  Then $\sqw{M_{|f|}}{g(\Delta)} \in \mathcal{L}^1(L^2(\TT^n))$ if and only if $f \in L^1(\TT^n)$ and
$$
\Tr(\sqw{M_f}{g(\Delta)}) := \Tr(g(\Delta)) \int_{\TT^n} f(\mathbf{x}) d^n \mathbf{x} \ , \ \fa f \in L^1(\TT^n) .
$$
\end{cor}

\begin{cor} \label{cor:4.3}
Let $0 < G(\Delta) \in \mathcal{L}^{1,\infty}(L^2(\TT^n))$ be measurable.
Then we have $\sqw{M_{|f|}}{G(\Delta)^s} = G(\Delta)^{s/2} M_{|f|} G(\Delta)^{s/2} \in \mathcal{L}^{1}(L^2(\TT^n))$ for all $s > 1$
if and only if $f \in L^1(\TT^n)$.  Moreover, setting
$$
\psi_\bl(M_f) := \bl \left( \frac{1}{k} \Tr(\sqw{M_f}{G(\Delta)^{1+\frac{1}{k}}}) \right) \ , \ \fa f \in L^1(\TT^n)
$$
for any $\bl \in BL$,
$$
\psi_\bl(M_f) := \lim_{k \to \infty} \frac{1}{k}
\Tr(\sqw{M_f}{G(\Delta)^{1+\frac{1}{k}}}) = c \int_{\TT^n} f(\mathbf{x}) d^n \mathbf{x} \ , \ \fa f \in L^1(\TT^n)
$$
for a constant $c \geq 0$ independent of $\bl \in BL$.
\end{cor}
\begin{proof}
From Corollary \ref{cor:4.2} it follows
$$
\lim_{k \to \infty} k^{-1} \Tr(\sqw{M_f}{G(\Delta)^{1+k^{-1}}}) 
= \lim_{k \to \infty} k^{-1} \Tr(G(\Delta)^{1+k^{-1}}) \int_{\TT^n} f(\mathbf{x}) d^n \mathbf{x}.
$$
As in Corollary \ref{cor:3.3}, set
$c = \lim_{k \to \infty} k^{-1} \Tr(G(\Delta)^{1+k^{-1}})$.
\end{proof}

\subsubsection*{Proof of Theorem \ref{thm:1b2}}

From Example~\ref{ex:compact_M} we have the bound
$$
c_s \nm{f}_1 \leq \nm{f}_{1,\mu_s} \leq C \nm{f}_1 .
$$
Since $\mu_s = \mu_g$ for $g = (1+x^2)^{-ns/2}$, it
follows from Proposition~\ref{prop:4.1} that
$T_\Delta^{s/2} M_f T_\Delta^{s/2} \in \mathcal{L}^1(L^2(M))$
for all $s > 1$ if and only if $f \in L^1(M)$.

Now, let $L^\infty(M) \ni f_n \to f \in L^1(M)$.  By the above
bound
$$
\Tr(T_\Delta^{s/2} M_{|f-f_n|} T_\Delta^{s/2})
\leq C \nm{f-f_n}_1 .
$$
Therefore
$$
\lim_{n \to \infty} \limsup_{s \to 1^+}
\Tr(T_\Delta^{s/2} M_{|f-f_n|} T_\Delta^{s/2}) = 0 .
$$
Hence
\begin{eqnarray*}
\lim_{s \to 1^+} \Tr(T_\Delta^{s/2} M_{f} T_\Delta^{s/2})
& = & \lim_{n \to \infty} \lim_{s \to 1^+} \Tr(T_\Delta^{s/2} M_{f_n} T_\Delta^{s/2}) \\
& = & \lim_{n \to \infty} \lim_{s \to 1^+} \Tr(M_{f_n} T_\Delta^s) \\
& = & \lim_{n \to \infty} \int_M f_n(x) |\mathrm{vol}|(x) \\
& = & \int_M f(x) |\mathrm{vol}|(x) .
\end{eqnarray*}
\qed

Corollary~\ref{cor:4.3} and Theorem~\ref{thm:1b2} shows that the residue of the zeta function  $\Tr(\sqw{M_f}{T_\Delta^s})$ at $s=1$ is an algebraic
expression that can be identified
with the Lebesgue integral of any integrable function.  The claim of (\cite{GBVF}, Cor 7.22), which must use the symmetrised expression
for any $f \in L^p(M)$, $1 \leq p < 2$, would be that $\Tr_\omega(\sqw{M_f}{T_\Delta}) = \Tr_\omega(T_\Delta^{1/2} M_f T_\Delta^{1/2})$
is also the Lebesgue integral of any integrable function.

The next result shows the claim is false. 

\begin{lemma} \label{lemma:6.3}
Let $\Delta$ be the Hodge Laplacian on the flat 1-torus $\TT$
and $T_\Delta := (1+\Delta)^{-1/2} \in \mathcal{L}^{1,\infty}(L^2(\TT))$.
There is a positive function $f \in L^1(\TT)$
such that the operator $T_{\Delta}^{1/2}M_{f}T_{\Delta}^{1/2}$
is not Hilbert-Schmidt.
\end{lemma}
\begin{proof}
  Fix~$\epsilon > 0$.  We use $\TT \cong [-\frac{1}{2},\frac{1}{2}]' = [-\frac{1}{2},\frac{1}{2}] \, / \! \! \sim$ where the endpoints are identified.  Consider the function $$ f(t) = \frac
  1 { \left| t \right|\, \left| \log |t| \right|^{1 + \epsilon}}. $$ The
  function~$f$ is clearly in $L^1([- \frac 12, \frac 12]')$.  We also
  consider the orthonormal system~$\left\{h_n \right\}_{n =
    1}^\infty$ given by $$ h_n (t) = 2^{n/2} \chi_n (t), $$
  where~$\chi_n$ is the characteristic function for~$2^{-n - 1} \leq \left|
    t \right| \leq 2^{-n}$.  Let us show that
  \begin{equation}
    \label{CFexampleObjective}
    \sum_{n = 1}^\infty \left|  \langle T(h_n), h_n \rangle  \right|^2
    = + \infty ,
  \end{equation}
  which in particular means that~$T:= M_{\sqrt{f}}(1+\Delta)^{-1/2}M_{\sqrt{f}}$ is not Hilbert-Schmidt, see (\cite{GohbergKrein}, Thm~4.3).  The operator~$T$ admits the
  following representation\footnote{The symbol~$x \otimes y$ stands
    for the one-dimensional operator defined by the functions~$x, y
    \in L^2([-\frac 12, \frac 12]')$.}
  \begin{equation}
    \label{Trepresentation}
    T = \sum_{k = - \infty}^{+\infty} \lambda_k \, \sqrt{f}  e_k \otimes
    \sqrt{f} e_k,
  \end{equation}
  where~$\lambda_k = (1 + 4\pi^2 k^2)^{-1/2}$ and~$e_k(t) =
  e^{2\pi i k t}$.

  We employ~(\ref{Trepresentation}) to show~(\ref{CFexampleObjective}).
   For the
  one-dimensional projection~$x \otimes x$, $x \in L^2([-\frac 12,
  \frac 12]')$, we have~$x \otimes x (y) = \langle y, x \rangle x$ for
  every~$y \in L^2([- \frac 12, \frac 12]')$).  Therefore $$ \langle x
  \otimes x (y), y \rangle = \left| \langle x, y \rangle \right|^2 =
  \left| \int_{-\frac 12}^{\frac 12} x(t) \overline{y(t)} \, dt \right|^2. $$
  Consequently,
  \begin{equation}
    \label{CFexampleTmp}
    \langle T(h_n), h_n \rangle = \sum_{k = - \infty}^{+\infty}
    \lambda_k \, \left| \int_{-\frac 12}^{\frac 12} \sqrt{f}(t)\,
       e_k(t)\, h_n(t)\, dt \right|^2. 
  \end{equation}
  In order to estimate the latter integral terms, let us observe that,
  for every~$\left| k \right| \leq 2^{n-3}$, $$ \cos \left( 2 \pi k t
  \right) \geq \frac 12,\ \ 2^{-n-1} \leq \left| t \right| \leq
  2^{-n}. $$ Consequently,
  \begin{multline*}
    \left| \int_{2^{-n-1} \leq \left| t \right| \leq 2^{-n}} \frac
      {2^{n/2}}{\left| t \right|^{1/2} \left| \log |t| \right|^{\frac {1
            +\epsilon}2}} \, e^{2\pi i kt}\, dt \right|^2 \\ \geq
    \left[ \int_{2^{-n-1} \leq \left| t \right| \leq 2^{-n}} \frac
      {2^{n/2}}{\left| t \right|^{1/2} \left| \log |t| \right|^{\frac
          {1+\epsilon}2}} \, \cos \left( 2\pi kt \right)\, dt
    \right]^2 \\ \geq 2^{-n-2} \inf_{2^{-n-1} \leq \left| t \right|
      \leq 2^{-n}} \frac 1 {\left| t \right|\, \left| \log |t|
      \right|^{{1+\epsilon}}} \geq \frac {c_0} {n^{1+\epsilon}},
  \end{multline*}
  for some numerical constant~$c_0 > 0$.  Returning
  to~(\ref{CFexampleTmp}), we see that, for another numerical
  constant~$c_1> 0$, $$ \langle T(h_n), h_n \rangle \geq \frac
  {c_0}{n^{1+\epsilon}} \sum_{\left| k \right| \leq 2^{n-3}}
  {\lambda_k} = \frac {c_0}{n^{1+\epsilon}} \sum_{\left| k \right|
    \leq 2^{n-3}} \frac 1{(1 + 4 \pi^2 k^2)^{\frac 12}} \geq \frac
  {c_1}{n^\epsilon}. $$ From the latter, it clearly follows that the
  series in~(\ref{CFexampleObjective}) diverges for~$\epsilon \leq
  \frac 12$.  It follows $(1+\Delta)^{-1/4}M_{f}(1+\Delta)^{-1/4}$ is not Hilbert-Schmidt by Lemma \ref{lemma:sym}.
\end{proof}

\begin{rems}
It was shown in (\cite{CRSS}, Thm 4.5 p.~266) that
$$
\limsup_{s \to 1^+} (s-1)\Tr(T^{s}) < \infty \Rightarrow 0 < T \in \mathcal{L}^{1,\infty}.
$$
From the first display in the proof of (\cite{CPS}, Prop 3.6 p.~88)
\begin{multline*}
\limsup_{s \to 1^+} (s-1)\Tr(\sqrt{A}T^{s}\sqrt{A}) \\
=  \limsup_{s \to 1^+} (s-1)\Tr((\sqrt{A}T\sqrt{A})^s) < \infty \\
 \Rightarrow 0 < \sqrt{A}T\sqrt{A} \in \mathcal{L}^{1,\infty}
\end{multline*}
for all \emph{bounded} positive operators $0 < A \in B(H)$.
Lemma \ref{lemma:6.3}, in combination with Corollary \ref{cor:4.3}, provides an example where this implication fails for $T \in \mathcal{L}^{1,\infty}$ and $\sqrt{A}$ an \emph{unbounded} positive linear operator.  In particular, from Lemma \ref{lemma:6.3}, we have an example where
$\sqrt{A}T\sqrt{A} \not\in \mathcal{L}^{1,\infty}$ and hence 
$$
\limsup_{s \to 1^+} (s-1)\Tr((\sqrt{A}T\sqrt{A})^{s}) = \infty ,
$$
yet, from Corollary \ref{cor:4.3},
$$
\limsup_{s \to 1^+} (s-1)\Tr(\sqrt{A}T^s\sqrt{A}) < \infty .
$$
\end{rems}

Our final result is that the failure of the symmetrised Dixmier trace formula on the torus is pointed at $L^1(\TT)$.  

\begin{thm}
  \label{TCLonePlusEpsilon}
  Let~$0 < G(\Delta) \in \mathcal{L}^{1, \infty}(L^2(\TT^n))$ be measurable
  and~$f \in L^{1 + \epsilon} (\TT^n)$ for $\epsilon > 0$.  Then~$\sqw{M_f}{G(\Delta)} = G(\Delta)^{1/2}
  M_f G(\Delta)^{1/2} \in \mathcal{L}^{1, \infty}(L^2(\TT^n))$ and
 $$
\Tr_\omega (\sqw{M_f}{G(\Delta)}) =  \Tr_\omega (G(\Delta)^{1/2} M_f
  G(\Delta)^{1/2}) = c \int_{\TT^n} f(\mathbf{x}) \, d^n \mathbf{x} \ , \ \fa f \in L^{1+\epsilon}(\TT^n)
 $$ for a constant~$ 0 \leq c =
  \Tr_\omega (G(\Delta))$ independent of~$\omega \in DL_2$.
\end{thm}
\begin{proof}
  Let~$R$ be the von Neumann algebra generated by the spectral projections of~$\Delta$.  Note that
  the subspace~$R \cap E$ is complemented in~$E$, for every symmetric
  ideal~$E$ of compact operators.  Note also that the subspace~$R \cap E$ is isomorphic to the sequence space~$\ell_E$.

  Let us now consider the bilinear operator $$ T(f, G) = M_f G,\
  \ f \in L^2(\TT^n),\ G \in R \cap \mathcal{L}^\infty. $$
  Here $\mathcal{L}^\infty$ denotes the bounded operators.  The
  following relations establish the boundedness of the
  operator~$T$ with different combinations of spaces
  \begin{align}
    T: &\, L^\infty(\TT^n) \times \mathcal{L}^\infty \mapsto \mathcal{L}^\infty,\ \
    \| T (f, G) \|_\infty \leq \| f \|_\infty \,
    \| G \|_{\infty} \label{BddRelation} \\
    T: &\, L^2(\TT^n) \times \mathcal{L}^2 \mapsto \mathcal{L}^2,\ \ \| T (f, G)
    \|_2 \leq \| f \|_2 \,
    \|  G \|_2 . \label{Prop42corollary}
  \end{align}
  Relation~(\ref{BddRelation}) is evident
  and~(\ref{Prop42corollary})~follows from Proposition~\ref{prop:2.1}.
  Applying bilinear complex interpolation,
  see~(\cite{BergLofstrom}, Thm~4.4.1), to the pair of
  relations~(\ref{BddRelation}) and~(\ref{Prop42corollary}) yields
  \begin{equation}
    \label{InterpolationI}
    \| M_f G \|_p \leq \| f \|_{p} \,
    \| G \|_p,\ \ f \in L^p(\TT^n),\ G \in R \cap \mathcal{L}^p,\
    2 \leq p \leq \infty. 
  \end{equation}
  Furthermore, it follows from the proof of Corollary~\ref{cor:2.3}
  that
  \begin{equation}
    \label{InterpolationII}
    \| M_f G \|_p \leq \| f \|_2
    \, \| G \|_p,\ \ f \in L^2(\TT^n),\ G \in R \cap
    \mathcal{L}^p,\ 1 < p \leq 2.  
  \end{equation}

  Let us fix positive~$f \in L^{1 + \epsilon} (\TT^n)$.
  We also
  fix~$0 < G(\Delta) \in \mathcal{L}^{1, \infty}$ and a
  factorization~$f = f_1 f_2$ such that $$ \| f \|_{1+
    \epsilon} = \| f_1 \|_{2 + \epsilon_1} \|
    f_2 \|_{2}, $$ for some~$\epsilon_1 > 0$.
  
  Let us fix numbers~$s, s_1, s_2 > 1$ such that~$s^{-1} = s^{-1}_1 +
  s^{-1}_2$ and~$2s < s_1 < 2 + \epsilon_1$, $s_2 < 2$.  Such numbers
  can always be found if~$s$ is sufficiently close to~$1$.  Finally, set
  $$
  G_1 = G(\Delta)^{s/s_1} \ \ \text{and}\ \ G_2 = G(\Delta)^{s/s_2}. $$
  
  Now we can estimate $$ \| G_1 M_f G_2 \|_s \leq \| G_1
    M_{f_1} \|_{s_1} \| M_{f_2} G_2 \|_{s_2} \leq
  \| f_1 \|_{s_1} \, \| G_1 \|_{s_1}
  \| f_2 \|_{2} \, \| G_2 \|_{s_2}, $$ where
  the last estimate is due to~(\ref{InterpolationI})
  and~(\ref{InterpolationII}).  Furthermore, since~$\| f_1
  \|_{s_1} \leq \| f_1 \|_{2 + \epsilon_1}$, we
  obtain $$
\| G_1 M_f G_2 \|_s \leq \| f \|_{1
    + \epsilon} \, \| G(\Delta)^s \|_1^{1/s_1} \, \| G(\Delta)^s
  \|_1^{1 / s_2} = \| f \|_{1 + \epsilon} \, \|
    G(\Delta) \|_s. $$    
Set $f_N(x) := f(x) \chi_{\inset{y}{f(y) \leq N}}(x)$, $N \in \NN$.  Then
$\| G(\Delta)^{1/2} M_{f_N} G(\Delta)^{1/2} \|_s \leq \| G_1 M_{f_N} G_2 \|_s$ by an application of Lemma~\ref{IIIlineLemma} using
$\theta = 1 - 2s/s_1$.
Using the noncommutative Fatou Lemma, (\cite{S}, Thm 2.7(d)),
\begin{eqnarray*}
\| G(\Delta)^{1/2} M_{f} G(\Delta)^{1/2} \|_s
& \leq & \sup_N \| G(\Delta)^{1/2} M_{f_N} G(\Delta)^{1/2} \|_s \\
& \leq & \sup_N \| f_N \|_{1 + \epsilon} \, \|G(\Delta) \|_s
    = \| f \|_{1 + \epsilon} \, \|  G(\Delta) \|_s .
\end{eqnarray*}
Finally, recalling from (\ref{eq:Z_1}) that $$\| G(\Delta) \|_{Z_1} = \limsup_{s \to 1^+} (s - 1) \, \| G(\Delta) \|_s, $$ we arrive at
 \begin{equation} \label{eq:epsilonest}
    \|
    G(\Delta)^{1/2} M_f G(\Delta)^{1/2} \|_{Z_1} \leq
    \| f \|_{1 + \epsilon} \, \| G(\Delta)
  \|_{Z_1}.
 \end{equation}
It follows that $G(\Delta)^{1/2} M_f G(\Delta)^{1/2} \in \mathcal{L}^{1,\infty}$
from (\cite{CRSS}, Thm 4.5).

  The trace identity follows from (\ref{eq:epsilonest}) and Corollary \ref{cor:3.3}.  In particular, take $L^\infty(\TT^n) \ni f_N \nearrow f \in L^{1+\epsilon}(\TT^n)$ as above with $\nm{f-f_N}_{1+\epsilon} \to 0$ as $N \to \infty$ by the Monotone Convergence Theorem.
Then $|\Tr_\omega(G(\Delta)^{1/2} M_{f - f_N} G(\Delta)^{1/2})| \leq e \left\| f -f_N \right\|_{1 + \epsilon} \, \left\| G(\Delta)  \right\|_{Z_1} \to 0$ as $N \to \infty$ by (\ref{eq:epsilonest}) and the fact $\nm{\cdot}_0 \leq e \nm{\cdot}_{Z_1}$ (\cite{CRSS}, Thm 4.5).  Employing Corollary \ref{cor:3.3}
  for $M_{f_N} \in B(L^2(M))$,
  \begin{eqnarray*}
  \Tr_\omega(G(\Delta)^{1/2} M_f G(\Delta)^{1/2}) & = & \lim_{N \to \infty}   \Tr_\omega(G(\Delta)^{1/2} M_{f_N} G(\Delta)^{1/2}) \\
  & \stackrel{\text{(Lemma \ref{lemma:sym})}}{=} & \lim_{N \to \infty}  \Tr_\omega(M_{f_N}^{1/2} G(\Delta) M_{f_N}^{1/2}) \\
  & = & \lim_{N \to \infty}  \Tr_\omega(M_{f_N} G(\Delta)) \\
  & \stackrel{\text{(Cor \ref{cor:3.3})}}{=} & c \lim_{N \to \infty} \int_{\TT^n} f_N(\mathbf{x}) \, d^n \mathbf{x} \\
  & = & c \int_{\TT^n} f(\mathbf{x}) \, d^n \mathbf{x} .
  \end{eqnarray*}
Recall that $f$ was positive.  By linearity, the result follows
for all $f \in L^{1+\epsilon}(\TT^n)$. 
\end{proof}

\begin{lemma}
  \label{IIIlineLemma}
  If $0 < B \in B(H)$ and $A = A^* \in B(H)$,
  then $$ \left\| B^{1/2} A
    B^{1/2} \right\|_E \leq \left\| B^{1/2 - \theta /2} A B^{1/2 +
      \theta/2} \right\|_E,\ \ 0 < \theta < 1. $$ Here~$E$ is a
  symmetric ideal of compact operators with symmetric norm $\| \cdot \|_E$.
\end{lemma}
\begin{proof}
  It was proven in~(\cite{CaPoSu}, Lemma 25) that, for positive
  bounded operators~$B_0, B_1$ and a bounded operator~$C$, the following
  estimate is valid $$ \| B_0^{1/2} C B_1^{1/2} \|_E \leq
  \, \left\| B_0 C \right\|_E^{1/2} \, \left\| C B_1
  \right\|^{1/2}_E. $$ Now, the lemma follows if we apply the estimate
  above to the operators $$ C = B^{1/2 - \theta /2} A B^{1/2 - \theta
    /2} \ \ \text{and}\ \ B_0 = B_1 = B^\theta, $$ and observe
  that~$A$ is selfadjoint.
\end{proof}




\begin{thebibliography}{10}
\expandafter\ifx\csname url\endcsname\relax
  \def\url#1{\texttt{#1}}\fi
\expandafter\ifx\csname urlprefix\endcsname\relax\def\urlprefix{URL }\fi
\expandafter\ifx\csname href\endcsname\relax
  \def\href#1#2{#2} \def\path#1{#1}\fi

\bibitem{S}
B.~Simon, Trace Ideals and their Applications, no. 120 in Mathematical Surveys
  and Monographs, AMS, 2005.

\bibitem{BR}
O.~Bratteli, D.~W. Robinson, Operator Algebras and Quantum Statistical
  Mechanics 1, 2nd Edition, Texts and Monographs in Physics, Springer-Verlag,
  Heidelberg, 1987.

\bibitem{Ped}
G.~K. Pederson, C$^*$-algebras and their automorphism groups, no.~14 in LMS
  Monographs, Academic Press, London, 1979.

\bibitem{Dix}
J.~Dixmier, Existence de traces non normales, C.~R.~Acad.~Sci.~Paris~(262)
  (1966) A1107--A1108.

\bibitem{Se2}
I.~E. Segal, A non-commutative extension of abstract integration, Ann. Math. 57
  (1953) 401--457.

\bibitem{T}
M.~Takesaki, Theory of Operator Algebras I, no. 124 in Encyclopaedia of
  Mathematical Sciences (Operator Algebras and Non-Commutative Geometry V),
  Springer-Verlag, Berlin, 2002.

\bibitem{C3}
A.~Connes, The {A}ction {F}unctional in {N}on-{C}ommutative {G}eometry, Comm.
  Math. Phys. 117 (1988) 673--683.

\bibitem{CN}
A.~Connes, Noncommutative Geometry, Academic Press, New York, 1994.

\bibitem{CM}
A.~Connes, H.~Moscovici, The {L}ocal {I}ndex {F}ormula in {N}oncommutative
  {G}eometry, Geometric and Functional Analysis 5 (1995) 174--243.

\bibitem{GBVF}
J.~M. Gracia-Bond{\'{i}}a, J.~C. V{\'{a}}rilly, H.~Figueroa, Elements of
  Noncommutative Geometry, Birkh{\"{a}}user Advanced Texts, Birkh{\"{a}}user,
  Boston, 2001.

\bibitem{Hawk}
E.~Hawkins, Hamiltonian gravity and noncommutative geometry, Comm. Math. Phys.
  187~(2) (1997) 471--489.

\bibitem{Wod}
M.~Wodzicki, Local invariants of spectral assymmetry, Invent. Math. 75 (1984)
  143--178.

\bibitem{FB}
M.-T. Benameuar, T.~Fack, Type {II} non-commutative geometry. {I}. {D}ixmier
  trace in von {N}eumann algebras, Adv. Math. 199 (2006) 29--87.

\bibitem{Landi}
G.~Landi, An Introduction to Noncommutative Spaces and Their Geometries, no.~51
  in Lecture Notes in Physics, Springer-Verlag, Berlin, 1998.

\bibitem{Fack}
T.~Fack,
 \href{http://www.math.leidenuniv.nl/~mdejeu/NoncomIntWeek_2008_Fack_singular%
_traces_in_NCG.pdf}{Singular {T}races in {N}.{C}.{G}.}, Educational Week on
  Noncommutative Integration, Thomas Stieltjes Institute for Mathematics,
  \url{http://www.math.leidenuniv.nl/~mdejeu/NoncomIntWeek_2008_Fack_singular_%
traces_in_NCG.pdf} (Jun. 2008).
\newline\urlprefix\url{http://www.math.leidenuniv.nl/~mdejeu/NoncomIntWeek_200%
8_Fack_singular_traces_in_NCG.pdf}

\bibitem{Rudin}
W.~Rudin, Invariant means on {L}$^\infty$, Studia Math. 44 (1972) 219--227.

\bibitem{AF}
A.~Carey, F.~A. Sukochev, Dixmier traces and some applications in
  non-commutative geometry, Russ. Math. Surv. 61 (2006) 1039--1099.

\bibitem{CRSS}
A.~Carey, A.~Rennie, A.~Sedaev, F.~A. Sukochev, The {D}ixmier trace and
  asymptotics of zeta functions, J. Funct. Anal. 249 (2007) 253--283.

\bibitem{CPS}
A.~Carey, J.~Phillips, F.~A. Sukochev, Spectral {F}low and {D}ixmier {T}races,
  Adv. Math. 173 (2003) 68--113.

\bibitem{LS2}
S.~Lord, F.~A. Sukochev, {Noncommutative
  residues and a characterisation of the noncommutative integral},
  accepted Proc. Amer. Math. Soc., to appear.

\bibitem{LSS}
S.~Lord, A.~Sedaev, F.~A. Sukochev, {D}ixmier {T}races as {S}ingular
  {S}ymmetric {F}unctionals and {A}pplications to {M}easurable {O}perators, J.
  Funct. Anal. 244~(1) (2005) 72--106.

\bibitem{GIV}
V.~Gayral, B.~Iochum, J.~V{\'{a}}rilly, Dixmier traces on noncompact
  isospectral deformations, J. Funct. Anal. 237 (2006) 507--539.

\bibitem{Bik}
A.~M. Bikchentaev, On a {P}roperty of ${L}_p$ {S}paces on {S}emifinite von
  {N}eumann algebras, Math. Notes 64~(1-2) (1998) 159--163.

\bibitem{RS}
M.~Reed, B.~Simon, Methods of Modern Mathematical Physics, Vol. I: Functional
  Analysis, revised and enlarged edition, Academic Press, San Diego-London,
  1980.

\bibitem{Bik98}
A.~Bikchentaev, {Majorization for Products of Measurable Operators}, Internat.
  J. Theoret. Phys. 37~(1) (1998) 571--576.

\bibitem{GohbergKrein}
I.~C. Gohberg, M.~G. Kre{\u\i}n, Vvedenie v teoriyu lineinykh
  nesamosopryazhennykh operatorov v gilbertovom prostranstve, Izdat. ``Nauka'',
  Moscow, 1965.

\bibitem{BergLofstrom}
J.~Bergh, J.~L{\"{o}}fstr{\"{o}}m, Interpolation spaces, an introduction, no.
  223 in Grundelehren Math. Wiss., Springer-Verlag, Berlin-Heidelberg-New York,
  1976.

\bibitem{CaPoSu}
A.~Carey, D.~Potapov, F.~Sukochev,
  \href{http://arxiv.org/abs/0807.2129v2}{Spectral flow is the integral of one
  forms on the {B}anach manifold of self adjoint {F}redholm operators},
  \href{http://arxiv.org/abs/0807.2129v2}{arXiv:0807.2129v2} [math.FA] (Jul.
  2008).
\newline\urlprefix\url{http://arxiv.org/abs/0807.2129v2}

\end{thebibliography}
\end{document}